\documentclass[reqno,10pt,centertags]{amsart} 
\usepackage{amsmath,amsthm,amscd,amssymb,latexsym,upref,enumerate,mathtools,mathrsfs}
\usepackage{bbm}
\usepackage{nicefrac}
\usepackage[usenames,dvipsnames,svgnames,table]{xcolor}

\DeclareFontFamily{U}{mathx}{}
\DeclareFontShape{U}{mathx}{m}{n}{<-> mathx10}{}
\DeclareSymbolFont{mathx}{U}{mathx}{m}{n}
\DeclareMathAccent{\widehat}{0}{mathx}{"70}
\DeclareMathAccent{\widecheck}{0}{mathx}{"71}

\usepackage{hyperref} 
\newcommand*{\mailto}[1]{\href{mailto:#1}{\nolinkurl{#1}}}
\newcommand{\arxiv}[1]{\href{http://arxiv.org/abs/#1}{arXiv:#1}}

\newcommand{\bbC}{{\mathbb{C}}}

\newcommand{\bbN}{{\mathbb{N}}}

\newcommand{\bbR}{{\mathbb{R}}}
\newcommand{\bbT}{{\mathbb{T}}}

\newcommand{\cB}{{\mathcal B}}

\newcommand{\cD}{{\mathcal D}}

\newcommand{\cF}{{\mathcal F}}

\newcommand{\cH}{{\mathcal H}}

\newcommand{\cM}{{\mathcal M}}

\newcommand{\cS}{{\mathcal S}}

\newcommand{\cV}{{\mathcal V}}

\newcommand{\gga}{\mathfrak{a}}
\newcommand{\gb}{\mathfrak{b}}

\newcommand{\gs}{\mathfrak{s}}

\DeclareMathOperator{\ran}{ran}
\DeclareMathOperator{\dom}{dom}

\DeclareMathOperator{\tr}{tr}

\DeclareMathOperator*{\sgn}{sgn}

\renewcommand{\ln}{\text{\rm ln}}

\newcommand{\loc}{\operatorname{loc}}
\newcommand{\locunif}{\operatorname{loc \, unif}}

\DeclareMathOperator*{\LIM}{\text{l.i.m.}}

\newcommand{\no}{\notag}
\newcommand{\lb}{\label}
\newcommand{\f}{\frac}

\newcommand{\ol}{\overline}

\newcommand{\wti}{\widetilde}

\newcommand{\hatt}{\widehat} 
\newcommand{\dott}{\,\cdot\,}

\newcommand{\bi}{\bibitem}

\renewcommand{\le}{\leqslant}

\let\geq\geqslant
\let\leq\leqslant

\makeatletter
\def\theequation{\@arabic\c@equation}

\allowdisplaybreaks 
\numberwithin{equation}{section}

\newtheorem{theorem}{Theorem}[section]
\newtheorem{proposition}[theorem]{Proposition}
\newtheorem{lemma}[theorem]{Lemma}
\newtheorem{corollary}[theorem]{Corollary}
\newtheorem{definition}[theorem]{Definition}
\newtheorem{hypothesis}[theorem]{Hypothesis}

\theoremstyle{remark}
\newtheorem{remark}[theorem]{Remark}

\newtheorem{example}[theorem]{Example}

\begin{document}

\title[A Generalized Birman--Schwinger Principle]{A Generalized Birman--Schwinger Principle and Applications to One-Dimensional Schr\"odinger Operators With Distributional Potentials}

\author[F.\ Gesztesy]{Fritz Gesztesy}
\address{Department of Mathematics, 
Baylor University, Sid Richardson Bldg., 1410 S.\,4th Street, Waco, TX 76706, USA}
\email{\mailto{Fritz\_Gesztesy@baylor.edu}}
\urladdr{\url{https://math.artsandsciences.baylor.edu/person/fritz-gesztesy-phd}}

\author[R.\ Nichols]{Roger Nichols}
\address{Department of Mathematics (Dept.~6956), The University of Tennessee at Chattanooga, 
615 McCallie Avenue, Chattanooga, TN 37403, USA}
\email{\mailto{Roger-Nichols@utc.edu}}
\urladdr{\url{https://sites.google.com/mocs.utc.edu/rogernicholshomepage/home}}

\date{\today}
\thanks{To appear in {\it Funct. Anal. Appl.}}
\@namedef{subjclassname@2020}{\textup{2020} Mathematics Subject Classification}

\dedicatory{Dedicated, with great admiration, to the memory of Dima Yafaev $($1948--2024$)$}

\begin{abstract} 
Given a self-adjoint operator $H_0$ bounded from below in a complex Hilbert space $\cH$, the corresponding scale of spaces $\cH_{+1}(H_0) \subset \cH \subset \cH_{-1}(H_0) = [\cH_{+1}(H_0)]^*$, and a fixed $V\in \cB(\cH_{+1}(H_0),\cH_{-1}(H_0))$, we define the operator-valued map $A_V(\dott):\rho(H_0)\to \cB(\cH)$ by
\[
A_V(z):=-\big(H_0-zI_{\cH} \big)^{-1/2}V\big(H_0-zI_{\cH} \big)^{-1/2}\in \cB(\cH),\quad z\in \rho(H_0),
\]
where $\rho(H_0)$ denotes the resolvent set of $H_0$.  Assuming that $A_V(z)$ is compact for some $z=z_0\in \rho(H_0)$ and has norm strictly less than one for some $z=E_0\in (-\infty,0)$, we employ an abstract version of Tiktopoulos' formula to define an operator $H$ in $\cH$ that is formally realized as the sum of $H_0$ and $V$.  We then establish a Birman--Schwinger principle for $H$ in which $A_V(\dott)$ plays the role of the Birman--Schwinger operator:  $\lambda_0\in \rho(H_0)$ is an eigenvalue of $H$ if and only if $1$ is an eigenvalue of $A_V(\lambda_0)$.  Furthermore, the geometric (but not necessarily the algebraic) multiplicities of $\lambda_0$ and $1$ as eigenvalues of $H$ and $A_V(\lambda_0)$, respectively, coincide.  

As a concrete application, we consider one-dimensional Schr\"odinger operators with $H^{-1}(\bbR)$ distributional potentials.
\end{abstract}

\maketitle

{\scriptsize{\tableofcontents}}

\maketitle



\section{Introduction} \lb{s1}

The case of general (i.e., three-coefficient) singular Sturm--Liouville operators, including distributional potentials, has been studied by Bennewitz and Everitt \cite{BE83} in 1983 (see also \cite[Sect.\ I.2]{EM99}). An extremely thorough and systematic investigation, including even and odd higher-order operators defined in terms of appropriate quasi-derivatives, and in the general case of matrix-valued coefficients (including distributional potential coefficients in the context of Schr\"odinger-type operators) was presented by Weidmann \cite{We87} in 1987. However, it was not until 1999 that Neiman-zade and Shkalikov \cite{NS99} and Savchuk and Shkalikov \cite{SS99} (see also \cite{SS03}) started a new development for Sturm--Liouville (resp., Schr\"odinger) operators with distributional potential coefficients in connection with areas such as, self-adjointness proofs, spectral and inverse spectral theory, oscillation properties, spectral properties in the non-self-adjoint context, etc. While any collection of pertinent references will still be utterly incomplete, we refer to the extensive bibliography on this subject until about 2014 in \cite{EGNT13}--\cite{EGNST15} and to \cite{ALZ23}--\cite{BS09a}, \cite{BW24}--\cite{Bo23}, \cite{DM09}--\cite{EGNST15}, \cite{FHMP09}, \cite{GN22}--\cite{GMP13}, \cite{Gu19}, \cite{HH14}--\cite{HMP11a},  \cite{KMS18}--\cite{Kr17a}, \cite{LS14}--\cite{LSW24},  \cite{MGM22}--\cite{RW24}, \cite{RSY24}, \cite{Sa10}--\cite{Sh14}, \cite{YS14}. For some multi-dimensional second-order operators with distributional potential coefficients we also refer, for instance, to \cite{He89}, \cite{MV02a}--\cite{MV06}, \cite{NS06}. 

It should also be mentioned that some of the attraction in connection with distributional potential coefficients 
in the Schr\"odinger operator stems from the low-regularity investigations of solutions of the 
Korteweg--de Vries (KdV) equation. We mention, for instance, \cite{BK15}, \cite{GR14}, \cite{KM01}, \cite{KT05}, \cite{KT06}, \cite{KV19}, \cite{KVZ18}, and \cite{Ry10}. 

The particular case of point interactions as special distributional coefficients in Schr\"odinger operators received enormous attention, too numerous to be mentioned here in detail. Hence, we primarily refer to the standard monographs \cite{AGHKH05} and \cite{AK01}.

The two main topics in this paper are an abstract Birman--Schwinger principle for a Birman--Schwinger operator that factors over a scale of Hilbert spaces associated to a self-adjoint operator and its application to one-dimensional Schr\"odinger operators with distributional potentials in $H^{-1}(\bbR)$.  These topics are individually treated in Sections \ref{s2} and \ref{s4}, respectively.  More specifically, Section \ref{s2} develops an abstract Birman--Schwinger principle employing scales of Hilbert spaces $\cH_{+1}(H_0) \subset \cH \subset \cH_{-1}(H_0) = [\cH_{+1}(H_0)]^*$ (also called Gelfand triples) associated with self-adjoint operators $H_0$ in $\cH$ bounded from below, by mimicking the concrete case of Schr\"odinger operators with appropriate potentials in $\cS'(\bbR)$.  Assuming for simplicity that $H_0$ is nonnegative and that $V$ is a bounded linear operator from $\cH_{+1}(H_0):=\dom(H_0^{1/2})$ to $\cH_{-1}(H_0)$, we introduce the Birman--Schwinger operator
\begin{equation}
A_V(z) = -\big(H_0-zI_{\cH}\big)^{-1/2}V\big(H_0-zI_{\cH}\big)^{-1/2},\quad z\in \rho(H_0),
\end{equation}
in $\cH$.  Assuming that $A_V(E_0)$ is compact with norm strictly less than one for some $E_0\in (-\infty,0)$, in Theorem \ref{t5.9} we prove that
\begin{equation}\lb{1.2}
\begin{split}
R(z) := \big(H_0 - zI_{\cH} \big)^{-1/2}[I_{\cH}-A_V(z)]^{-1}\big(H_0 - zI_{\cH} \big)^{-1/2},&\\
z\in \{\zeta\in \rho(H_0)\,|\, 1\in \rho(A_V(\zeta))\}=\rho(H_0)\backslash\cD_{H_0,V},&
\end{split}
\end{equation}
where $\cD_{H_0,V}\subset\bbC$ is a discrete set, is the resolvent of a densely defined operator $H$ in $\cH$.  In addition, the difference of resolvents of $H$ and $H_0$ is compact.  The definition of $R(z)$ in \eqref{1.2} is inspired by Tiktopoulos's formula (see, e.g., \cite[Section II.3]{Si71}). The operator $H$ may be viewed as a rigorous definition of the purely formal sum ``$H_0+V$.''  In Theorem \ref{t2.11c}, the main result in Section \ref{s2} and the primary abstract result of this paper, we show that a point $\lambda_0\in \rho(H_0)$ is an eigenvalue of $H$ if and only if $1$ is an eigenvalue of $A_V(\lambda_0)$.  Moreover, we establish a one-to-one correspondence between eigenvectors of $H$ corresponding to $\lambda_0$ and eigenvectors of $A_V(\lambda_0)$ corresponding to $1$.  In particular, the geometric multiplicities of $\lambda_0$ as an eigenvalue of $H$ and $1$ as an eigenvalue of $A_V(\lambda_0)$ coincide.

As an application of the abstract results developed in Section \ref{s2}, the case of one-dimensional Schr\"odinger operators with $H^{-1}(\bbR)$-potentials is treated in great detail in Section \ref{s4}.  In Proposition \ref{p2.9} we verify that the basic abstract assumptions of Section \ref{s2} hold when $V\in H^{-1}(\bbR)$ and $H_0$ is the self-adjoint realization of (minus) the second derivative in $L^2(\bbR)$.  In this case, $\cH_{\pm1}(H_0)$ are the Sobolev spaces $H^{\pm1}(\bbR)$, respectively.  The Birman--Schwinger operator $A_V(\dott)$ is actually trace class, hence Hilbert--Schmidt, and we calculate the Hilbert--Schmidt norm of $A_V(-\kappa^2)$ for $\kappa\in (0,\infty)$.  Section \ref{s4} culminates in Theorem \ref{t4.6c}, a Birman--Schwinger principle for one-dimensional Schr\"odinger operators with potentials $V\in H^{-1}(\bbR)$.  We close Section \ref{s4} by discussing an alternative Fourier space approach to some of the main calculations carried out in the proof of Proposition \ref{p2.9}.  Finally, Appendix \ref{s0} recalls some useful results on Sobolev multipliers $V \in \cS(\bbR)$ mapping $H^1(\bbR)$ into $H^{-1}(\bbR)$.

\medskip

\noindent 
{\bf Notation.} The symbols $\widehat f$ and $\widecheck f$ denote the Fourier and inverse Fourier transforms, respectively, of appropriate functions $f$; see \eqref{2.1} for the precise definitions.  The symbol $f\ast g$ denotes the convolution of an appropriate pair of functions $f$ and $g$; see \eqref{3.2c}.  The inner product in a separable (complex) Hilbert space $\cH$ is denoted by $(\,\cdot\,,\,\cdot\,)_{\cH}$ and is assumed to be linear with respect to the second argument.   If $T$ is a linear operator mapping (a subspace of) a Hilbert space into another, then $\dom(T)$ denotes the domain of $T$.  The resolvent set and spectrum of a closed linear operator in $\cH$ will be denoted by $\rho(\,\cdot\,)$ and $\sigma(\, \cdot \,)$, respectively.  The Banach space of bounded linear operators on $\cH$ is denoted by $\cB(\cH)$.  Finally, for $p\in [1,\infty)$, the corresponding $\ell^p$-based trace ideals will be denoted by $\cB_p (\cH)$ with norms abbreviated by $\|\,\cdot\,\|_{\cB_p(\cH)}$.

\section{A Generalized Birman--Schwinger Principle} \lb{s2}

In this section we derive a generalized Birman--Schwinger principle that is applicable to Schr\"odinger operators with distributional potential coefficients.

Let $\cH$ be a separable (complex) Hilbert space with inner product $(\dott,\dott)_{\cH}$, and assume that $H_0$ is a nonnegative self-adjoint operator in $\cH$ with domain $\dom(H_0)$.

\medskip
\noindent
{\bf Convention.}  We fix a branch of the complex square root function as follows:  If $z\in \bbC\backslash\{0\}$ with $z=|z|e^{i\arg(z)}$ and $\arg(z)\in [0,2\pi)$, then $z^{1/2} = |z|^{1/2}e^{i\arg(z)/2}$ (here $|z|^{1/2}$ denotes the nonnegative square root of $|z|$ and $0^{1/2}:=0$).  The branch so chosen is analytic in $\bbC\backslash[0,\infty)$ and continuous in $\bbC_+$ up to the cut along $[0,\infty)$.\hfill $\triangleleft$

\medskip

Let $\{E_{H_0}(\lambda)\}_{\lambda\in \bbR}$ denote the strongly right-continuous spectral family uniquely associated to $H_0$ by the spectral theorem, so that
\begin{equation}
H_0=\int_{\sigma(H_0)}\lambda\, dE_{H_0}(\lambda).
\end{equation}
Square roots are then defined as follows (cf., e.g., \cite[VI.5.2]{Ka80})
\begin{align}
\big(H_0-zI_{\cH}\big)^{1/2} &= \int_{\sigma(H_0)} (\lambda - z)^{1/2}\, dE_{H_0}(\lambda),\quad z\in \bbC,\lb{2.2e}\\
\big(H_0-zI_{\cH}\big)^{-1/2} &= \int_{\sigma(H_0)} \frac{1}{(\lambda - z)^{1/2}}\, dE_{H_0}(\lambda),\quad z\in \rho(H_0).
\end{align}
In particular, \eqref{2.2e} implies stability of square root domains in the following form:
\begin{equation}
\dom\big((H_0-zI_{\cH})^{1/2}\big) = \dom\big(H_0^{1/2}\big),\quad z\in \bbC.
\end{equation}

Introducing the Hilbert space $(\cH_{+1}(H_0), (\dott,\dott)_{\cH_{+1}(H_0)})$ via, 
\begin{align} 
& \cH_{+1}(H_0) = \dom\big(H_0^{1/2}\big),    \lb{5.1} \\
& (f,g)_{\cH_{+1}(H_0)} := \big(H_0^{1/2}f,H_0^{1/2}g\big)_{\cH} + (f,g)_{\cH},\quad f,g\in \cH_{+1}(H_0) = \dom\big(H_0^{1/2}\big),     \no 
\end{align}
one notes that
\begin{equation}
\|f\|_{\cH}\leq \|f\|_{\cH_{+1}(H_0)},\quad f\in \cH_{+1}(H_0).
\end{equation}
Since $H_0^{1/2}$ is self-adjoint and hence closed in $\cH$, $\cH_{+1}(H_0)$ equipped with the inner product $(\dott,\dott)_{\cH_{+1}(H_0)}$ defined in \eqref{5.1} is a Hilbert space.  Let $\cH_{-1}(H_0):=[\cH_{+1}(H_0)]^*$ denote the conjugate dual space of bounded conjugate linear functionals on $\cH_{+1}(H_0)$.

If $f\in \cH$, then the map $\Psi_f:\cH_{+1}(H_0) \to \bbC$ defined by
\begin{equation}\lb{5.3}
\Psi_f(g) = (g,f)_{\cH},\quad g\in \cH_{+1}(H_0),
\end{equation}
is conjugate linear and satisfies
\begin{equation}
|\Psi_f(g)| \leq \|f\|_{\cH}\|g\|_{\cH_{+1}(H_0)},\quad g\in \cH_{+1}(H_0).
\end{equation}
In particular, $\Psi_f\in \cH_{-1}(H_0)$.  Therefore, identifying $\Psi_f$ and $f$, $\cH$ may be regarded as a subspace of $\cH_{-1}(H_0)$ and
\begin{equation}
\|f\|_{\cH_{-1}(H_0)} = \big\|(H_0+I_{\cH})^{-1/2}f \big\|_{\cH}\leq \|f\|_{\cH},\quad f\in \cH.
\end{equation}
In addition, one has compatibility between the pairing 
${}_{\cH_{+1}(H_0)}\langle \dott,\dott_f\rangle_{\cH_{-1}(H_0)}$ and the scalar product $(\dott,\dott)_{\cH}$ in $\cH$ via
\begin{equation}
{}_{\cH_{+1}(H_0)}\langle g,\Psi_f\rangle_{\cH_{-1}(H_0)} = (g,f)_{\cH},\quad g\in \cH_{+1}(H_0),\, f\in \cH.
\end{equation}
A possible way to describe the Hilbert space $(\cH_{-1}(H_0), (\dott,\dott)_{\cH_{-1}(H_0)})$ is then given by 
\begin{align} 
& \cH_{-1}(H_0) = [\cH_{+1}(H_0)]^*,    \no \\
& (f,g)_{\cH_{-1}(H_0)} := \big((H_0 + I_{\cH})^{-1/2} f,(H_0 + I_{\cH})^{-1/2} g\big)_{\cH},     \lb{2.11} \\
& \hspace*{3.5cm} f,g\in \cH_{-1}(H_0) = [\cH_{+1}(H_0)]^*.     \no 
\end{align}

Thus, in this manner, one arrives at the inclusion (a scale) of spaces:
\begin{equation}\lb{5.7}
\cH_{+1}(H_0) \subset \cH \subset \cH_{-1}(H_0) = [\cH_{+1}(H_0)]^*, 
\end{equation}
in which
\begin{equation}
\begin{split}
\|f\|_{\cH}&\leq \|f\|_{\cH_{+1}(H_0)},\quad f\in \cH_{+1}(H_0),\\
\|g\|_{\cH_{-1}(H_0)}&\leq \|g\|_{\cH},\quad g\in \cH.
\end{split}
\end{equation}

If $z\in \rho(H_0)$, then $\big(H_0 - zI_{\cH} \big)^{-1/2}\in \cB(\cH)$.  In fact, $\big(H_0 - zI_{\cH} \big)^{-1/2}$ maps $\cH$ to $\cH_{+1}(H_0)$ as a bounded map between Hilbert spaces.

\begin{proposition}\lb{p5.1}
If $z\in \rho(H_0)$, then $\big(H_0 - zI_{\cH} \big)^{-1/2}\in \cB(\cH,\cH_{+1}(H_0))$ and
\begin{equation}
\big\|\big(H_0 - zI_{\cH} \big)^{-1/2} \big\|_{\cB(\cH,\cH_{+1}(H_0))} \leq M(z),
\end{equation}
where
\begin{equation}\lb{5.10z}
M(z):=\big\|H_0^{1/2}\big(H_0 - zI_{\cH} \big)^{-1/2} \big\|_{\cB(\cH)} + \big\|\big(H_0 - zI_{\cH} \big)^{-1/2}\big\|_{\cB(\cH)},\quad z\in \rho(H_0).
\end{equation}
\end{proposition}
\begin{proof}
Let $z\in \rho(H_0)$.  If $f\in \cH$, then $\big(H_0 - zI_{\cH} \big)^{-1/2}f\in \dom\big(H_0^{1/2}\big)= \cH_{+1}(H_0)$, and the definition of $\|\dott\|_{\cH_{+1}(H_0)}$ implies
\begin{align}
\big\|\big(H_0 - zI_{\cH} \big)^{-1/2}f\big\|_{\cH_{+1}(H_0)} \leq M(z) \|f\|_{\cH}.
\end{align}
\end{proof}

For $z\in \rho(H_0)$, the operator $\big(H_0 - zI_{\cH} \big)^{-1/2}$ may be extended from $\cH$ to $\cH_{-1}(H_0)$ as follows.  For $\Psi\in \cH_{-1}(H_0)$, consider the map $\wti \Psi:\cH\to \bbC$ defined by
\begin{equation}
\wti \Psi (g) = {}_{\cH_{+1}(H_0)}\big\langle \big(H_0 - \ol{z}I_{\cH} \big)^{-1/2}g,\Psi\big\rangle{}_{\cH_{-1}(H_0)},\quad g\in \cH.
\end{equation}
The map $\wti \Psi$ is a conjugate linear functional on $\cH$, and
\begin{equation}\lb{5.10}
\big|\wti \Psi (g)\big| \leq \|\Psi\|_{\cH_{-1}(H_0)}\big\|\big(H_0 - \ol{z}I_{\cH} \big)^{-1/2}g \big\|_{\cH_{+1}(H_0)}\leq M(\ol{z}) \|\Psi\|_{\cH_{-1}(H_0)}\|g\|_{\cH},\quad g\in \cH,
\end{equation}
where $M(\ol{z})$ is the constant defined by \eqref{5.10z}.

The inequality in \eqref{5.10} shows that $\wti \Psi\in \cH^*$.  By the Riesz representation theorem, $\wti \Psi$ may be identified with a vector $f_{\Psi}\in \cH$ such that $\big\|\wti \Psi\big\|_{\cH^*}=\|f_{\Psi}\|_{\cH}$.  Finally, one defines
\begin{equation}\lb{5.12}
\big(H_0 - zI_{\cH} \big)^{-1/2}\Psi = f_{\Psi}.
\end{equation}
Combining \eqref{5.10} and \eqref{5.12}, one obtains:
\begin{equation}\lb{5.14z}
\big\|\big(H_0 - zI_{\cH} \big)^{-1/2}\Psi\big\|_{\cH} \leq M(\ol{z})\|\Psi\|_{\cH_{-1}(H_0)},\quad \Psi\in \cH_{-1}(H_0).
\end{equation}
Thus, $\big(H_0 - zI_{\cH} \big)^{-1/2}\in \cB(\cH_{-1}(H_0),\cH)$.  Furthermore, one verifies that if $f\in \cH$ is identified with $\Psi_f\in \cH_{-1}(H_0)$ (cf.~\eqref{5.3}), then $\big(H_0 - zI_{\cH} \big)^{-1/2}\Psi_f = \big(H_0 - zI_{\cH} \big)^{-1/2}f$.  That is, the map $\big(H_0 - zI_{\cH} \big)^{-1/2}:\cH_{-1}(H_0)\to \cH$ defined above is an extension of the original map $\big(H_0 - zI_{\cH} \big)^{-1/2}:\cH\to \cH$ given by \eqref{2.2e}.  These considerations may be summarized as follows.

\begin{proposition}\lb{p5.2}
If $z\in \rho(H_0)$, then $\big(H_0 - zI_{\cH} \big)^{-1/2}:\cH_{-1}(H_0) \to \cH$ defined by \eqref{5.12} is an extension of $\big(H_0 - zI_{\cH} \big)^{-1/2}\in \cB(\cH)$.  Furthermore, this extension satisfies
\begin{equation}\lb{5.14}
\big(H_0 - zI_{\cH} \big)^{-1/2}\in \cB(\cH_{-1}(H_0),\cH)
\end{equation}
and
\begin{equation}\lb{5.15z}
\big\|\big(H_0 - zI_{\cH} \big)^{-1/2} \big\|_{\cB(\cH_{-1}(H_0),\cH)} \leq M(\ol{z}),\quad z\in \bbC\backslash [0,\infty).
\end{equation}
\end{proposition}
\begin{proof}
That $\big(H_0 - zI_{\cH} \big)^{-1/2}:\cH_{-1}(H_0) \to \cH$ defined by \eqref{5.12} is an extension of $\big(H_0 - zI_{\cH} \big)^{-1/2}\in \cB(\cH)$ satisfying \eqref{5.14} has already been shown.  The norm bound in \eqref{5.15z} immediately follows from \eqref{5.14z}.
\end{proof}

Propositions \ref{p5.1} and \ref{p5.2} yield the following improvement over the inclusion \eqref{5.7} in the form
\begin{equation}\lb{2.22}
\cH_{+1}(H_0) \hookrightarrow \cH = \cH^* \hookrightarrow \cH_{-1}(H_0) = [\cH_{+1}(H_0)]^*,
\end{equation}
where $\hookrightarrow$ abbreviates dense and continuous embedding between Hilbert (resp., Banach) spaces. Reflexivity of $\cH_{\pm 1}(H_0)$ also implies that 
\begin{equation}
 \cH_{+1}(H_0) = [\cH_{-1}(H_0)]^*.
\end{equation}

\begin{hypothesis}\lb{h5.4}
Let $\cH$ be a separable $($complex\,$)$ Hilbert space and $H_0\geq 0$ a self-adjoint operator in $\cH$.  Introduce the scale of spaces $\cH_{+1}(H_0) \subset \cH \subset\cH_{-1}(H_0)$ as in \eqref{5.7}.  Let $V\in \cB(\cH_{+1}(H_0),\cH_{-1}(H_0))$ and assume that
\begin{equation}\lb{5.22}
A_V(z):=-\big(H_0-zI_{\cH} \big)^{-1/2}V\big(H_0-zI_{\cH} \big)^{-1/2}\in \cB_{\infty}(\cH)
\end{equation}
for some $($hence, all\,$)$ $z\in \rho(H_0)$ and that
\begin{equation}\lb{5.23}
\big\|A_V(E_0)\big\|_{\cB(\cH)}<1
\end{equation}
for some $E_0\in (-\infty,0)$.
\end{hypothesis}

\begin{remark}
$(i)$  Assuming $A_V(z_0)\in \cB_{\infty}(\cH)$ for some $z_0\in \rho(H_0)$, one infers that for any other $z\in \rho(H_0)$,
\begin{align}
A_V(z)&= \overline{\big(H_0-zI_{\cH}\big)^{-1/2}\big(H_0-z_0I_{\cH}\big)^{1/2}}A_V(z_0)\no\\
&\quad\times\big(H_0-z_0I_{\cH}\big)^{1/2}\big(H_0-zI_{\cH}\big)^{-1/2}\no\\
&= \big(H_0-z_0I_{\cH}\big)^{1/2}\big(H_0-zI_{\cH}\big)^{-1/2}A_V(z_0)\lb{2.28u}\\
&\quad\times\big(H_0-z_0I_{\cH}\big)^{1/2}\big(H_0-zI_{\cH}\big)^{-1/2}\in\cB_{\infty}(\cH)\no
\end{align}
since
\begin{align}
&\big(H_0-z_0I_{\cH}\big)^{1/2}\big(H_0-zI_{\cH}\big)^{-1/2}\in \cB(\cH),\no\\
&\big(H_0-z_0I_{\cH}\big)^{1/2}\big(H_0-zI_{\cH}\big)^{-1/2}\in \cB(\cH),\\
&\big(H_0-z_0I_{\cH}\big)^{-1/2}V\big(H_0-z_0I_{\cH}\big)^{-1/2}\in \cB_{\infty}(\cH).\no
\end{align}
Thus, \eqref{5.22} extends from one to all $z\in \rho(H_0)$.\\[1mm]
$(ii)$ The assumption in Hypothesis \ref{h5.4} that $V\in \cB(\cH_{+1}(H_0),\cH_{-1}(H_0))$, combined with Propositions \ref{p5.1} and \ref{p5.2}, {\it a priori} yields $A_V(z)\in \cB(\cH)$ for $z\in \rho(H_0)$.  Therefore, \eqref{5.22} assumes the stronger condition that $A_V(z)$ is compact for $z\in \rho(H_0)$.~\hfill$\diamond$
\end{remark}

We recall the analytic Fredholm theorem in the following form (see, e.g., \cite[Theorem A.27]{Si71}):
\begin{theorem}[The Analytic Fredholm Theorem]\lb{t5.7}
Let $\cH$ be a separable $($complex\,$)$ Hilbert space and $\Omega$ an open connected subset of $\bbC$.  If $A(\dott):\Omega\to \cB_{\infty}(\cH)$ is analytic on $\Omega$, then the following alternative holds: Either\\[1mm]
$(i)$ $[I_{\cH}-A(z)]^{-1}$ does not exist for all $z\in \Omega$,\\[1mm]
or,\\[1mm]
$(ii)$ There is a discrete set $\cD\subset \Omega$ such that $[I_{\cH}-A(\dott)]^{-1}$ exists and is analytic on $\Omega\backslash \cD$ with poles at the discrete points of $\cD$.  If $z\in \cD$, then there exists $\phi\in \cH\backslash\{0\}$ such that $A(z)\phi=\phi$.
\end{theorem}

\begin{remark}
In the situation when $(ii)$ holds in Theorem \ref{t5.7}, the Fredholm alternative implies $[I-A(z)]^{-1}\in \cB(\cH)$, $z\in \rho(H_0)\backslash \cD$.  Therefore, in this case,
\begin{equation}
\{z\in \Omega\,|\, 1\in \rho(A(z))\} = \Omega\backslash \cD.
\end{equation}
\hfill $\diamond$
\end{remark}

\begin{proposition}\lb{p5.6}
Assume Hypothesis \ref{h5.4}.  Then the following statements $(i)$ and $(ii)$ hold:\\[1mm]
$(i)$ The map $A_V(\dott):\bbC\backslash[0,\infty)\to \cB_{\infty}(\cH)$ is analytic $($with respect to $\|\dott\|_{\cB(\cH)}$$)$ and is continuous in $\bbC_+$ up to points in $[0,\infty)\cap \rho(H_0)$.\\[1mm]
$(ii)$  There is a discrete set $\cD_{H_0,V} \subset \rho(H_0)$ such that $[I_{\cH}-A_V(\dott)]^{-1}$ exists and is analytic on $\rho(H_0)\backslash \cD_{H_0,V}$ with poles at the discrete points of $\cD_{H_0,V}$.  If $z\in \cD_{H_0,V}$, then there exists $g\in \cH\backslash\{0\}$ such that $A_V(z)g=g$.\\[1mm]
\end{proposition}
\begin{proof}
To establish the analyticity and continuity claims in $(i)$, one applies \eqref{2.28u} with $z_0=E_0$:
\begin{align}
A_V(z)&= \big(H_0-E_0I_{\cH}\big)^{1/2}\big(H_0-zI_{\cH}\big)^{-1/2}A_V(E_0)\lb{2.31u}\\
&\quad\times\big(H_0-E_0I_{\cH}\big)^{1/2}\big(H_0-zI_{\cH}\big)^{-1/2}\in\cB_{\infty}(\cH),\quad z\in \bbC\backslash[0,\infty).\no
\end{align}
Since the $z$-dependent factors on the right-hand side in \eqref{2.31u} are analytic functions of $z\in \bbC\backslash[0,\infty)$ and are continuous in $\bbC_+$ up to points in $[0,\infty)\cap \rho(H_0)$, the analyticity and continuity claims in $(i)$ follow.

By \eqref{5.23}, $[I_{\cH}-A_V(E_0)]^{-1}$ exists (and is, in fact, given by a norm-convergent geometric series in $\cB(\cH)$).  Thus, the assertions in $(ii)$ immediately follow by applying the analytic Fredholm theorem (i.e., Theorem \ref{t5.7}) to the map $A_V(\dott):\rho(H_0)\to \cB_{\infty}(\cH)$.
\end{proof}

Taking inspiration from Tiktopoulos's formula (see, e.g., \cite[Section II.3]{Si71}), one introduces
\begin{equation}\lb{5.26}
\begin{split}
R(z) := \big(H_0 - zI_{\cH} \big)^{-1/2}[I_{\cH}-A_V(z)]^{-1}\big(H_0 - zI_{\cH} \big)^{-1/2},&\\
z\in \{\zeta\in \rho(H_0)\,|\, 1\in \rho(A_V(\zeta))\}=\rho(H_0)\backslash\cD_{H_0,V}.&
\end{split}
\end{equation}
The conditions on $z$ in \eqref{5.26} ensure that $R(z)\in \cB(\cH)$.  In light of the elementary identity,
\begin{equation}\lb{5.29aa}
[I_{\cH} - A_V(z)]^{-1} = I_{\cH} + A_V(z)[I_{\cH} - A_V(z)]^{-1},\quad z\in \rho(H_0)\backslash\cD_{H_0,V},
\end{equation}
the operator $R(z)$ defined in \eqref{5.26} may be recast in the form
\begin{align}
R(z)&= (H_0-zI_{\cH})^{-1}\no\\
&\quad + (H_0-zI_{\cH})^{-1/2}A_V(z)[I_{\cH} - A_V(z)]^{-1}(H_0-zI_{\cH})^{-1/2},\lb{5.29a}\\
&\hspace*{6.25cm} z\in \rho(H_0)\backslash\cD_{H_0,V}.\no
\end{align}

\begin{theorem}\lb{t5.9}
Assume Hypothesis \ref{h5.4} and suppose that $z\in \rho(H_0)\backslash \cD_{H_0,V}$.  Then $R(z)$ defined by \eqref{5.26} uniquely defines a densely defined, closed, linear operator $H$ in $\cH$ by
\begin{equation}\lb{5.28}
R(z) = (H-zI_{\cH})^{-1}.
\end{equation}
Furthermore, $[\rho(H_0)\backslash\cD_{H_0,V}]\subseteq \rho(H)$ and $H$ has the property:
\begin{equation}\lb{5.29}
\big[(H-zI_{\cH})^{-1} - (H_0-zI_{\cH})^{-1}\big] \in \cB_{\infty}(\cH),\quad z\in \rho(H_0)\cap\rho(H).
\end{equation}
\end{theorem}
\begin{proof}
It suffices to prove that $R(z)$, $z\in \rho(H_0)\backslash\cD_{H_0,V}$, is the resolvent of a densely defined, closed, linear operator in $\cH$.  The compactness property in \eqref{5.29} follows for $z\in \rho(H_0)\backslash\cD_{H_0,V}$ from \eqref{5.29a} and compactness of $A_V(z)$.  In turn, compactness  extends to all $z\in \rho(H_0)\cap\rho(H)$ by applying the identity:
\begin{align}
&(H-zI_{\cH})^{-1}-(H_0-zI_{\cH})^{-1}\no\\
&\quad= (H-z_0I_{\cH})(H-zI_{\cH})^{-1}\big[(H-z_0I_{\cH})^{-1}-(H_0-z_0I_{\cH})^{-1} \big]\\
&\qquad\times(H_0-z_0I_{\cH})(H_0-zI_{\cH})^{-1},\quad z,z_0\in \rho(H_0)\cap\rho(H)\no
\end{align}
(see, e.g., \cite[Theorem C.4.8]{GNZ24}, \cite[p.~178]{We80} in this connection).

If $z\in \rho(H_0)\backslash\cD_{H_0,V}$, then $R(z)$ is injective since all three factors on the right-hand side in \eqref{5.26} are injective.  Furthermore, by taking adjoints in \eqref{5.26}, one obtains:
\begin{equation}
R(z)^* = (H_0 - \ol{z}I_{\cH} )^{-1/2}[I_{\cH}-A_V(z)^*]^{-1}(H_0 - \ol{z}I_{\cH} )^{-1/2},
\end{equation}
which implies $\{0\}=\ker(R(z)^*)=[\ran(R(z))]^{\perp}$.  Thus, $\ol{\ran(R(z))}=\cH$; that is, the range of $R(z)$ is dense in $\cH$.  To conclude that $R(z)$ is the resolvent of a densely defined, closed, linear operator in $\cH$, it suffices to verify that $R(\dott)$ satisfies the first resolvent identity in the form:
\begin{equation}\lb{5.34}
R(z_1)-R(z_2) = (z_1-z_2)R(z_1)R(z_2),\quad z_1,z_2\in \rho(H_0)\backslash\cD_{H_0,V}.
\end{equation}
To this end, let $z_1,z_2\in \rho(H_0)\backslash\cD_{H_0,V}$.  Employing the shorthand notation for the ``free'' resolvent, $R_0(z)=(H_0 - zI_{\cH} )^{-1}$, $z\in \rho(H_0)$, one applies \eqref{5.29a} and the first resolvent identity for $H_0$, to obtain
\begin{align} 
\begin{split} 
&R(z_1)-R(z_2) - (z_1-z_2)R(z_1)R(z_2)\lb{5.35} \\
&\quad= -R_0(z_2)^{1/2}A_V(z_2)[I_{\cH}-A_V(z_2)]^{-1}R_0(z_2)^{1/2}   \\
&\qquad + R_0(z_1)^{1/2}A_V(z_1)[I_{\cH}-A_V(z_1)]^{-1}R_0(z_1)^{1/2}   \\
&\qquad -(z_1-z_2)R_0(z_1)R_0(z_2)^{1/2}A_V(z_2)[I_{\cH}-A_V(z_2)]^{-1}R_0(z_2)^{1/2}    \\
&\qquad -(z_1-z_2)R_0(z_1)^{1/2}A_V(z_1)[I_{\cH}-A_V(z_1)]^{-1}R_0(z_1)^{1/2}R_0(z_2)    \\
&\qquad -(z_1-z_2)R_0(z_1)^{1/2}A_V(z_1)[I_{\cH}-A_V(z_1)]^{-1}R_0(z_1)^{1/2}    \\
&\qquad\quad\times R_0(z_2)^{1/2}A_V(z_2)[I_{\cH}-A_V(z_2)]^{-1}R_0(z_2)^{1/2}. 
\end{split} 
\end{align}
An application of the first resolvent equation for $R_0(\dott)$ implies
\begin{align}
&(z_1-z_2)R_0(z_1)R_0(z_2)^{1/2}A_V(z_2)[I_{\cH}-A_V(z_2)]^{-1}R_0(z_2)^{1/2}\no\\
&\quad= -(z_1-z_2)R_0(z_1)R_0(z_2)VR_0(z_2)^{1/2}[I_{\cH}-A_V(z_2)]^{-1}R_0(z_2)^{1/2}\lb{5.36}\\
&\quad= -[R_0(z_1)-R_0(z_2)]VR_0(z_2)^{1/2}[I_{\cH}-A_V(z_2)]^{-1}R_0(z_2)^{1/2}.\no
\end{align}
Similarly, one obtains
\begin{align}
&(z_1-z_2)R_0(z_1)^{1/2}A_V(z_1)[I_{\cH}-A_V(z_1)]^{-1}R_0(z_1)^{1/2}R_0(z_2)\no\\
&\quad=(z_1-z_2)R_0(z_1)^{1/2}[I_{\cH}-A_V(z_1)]^{-1}A_V(z_1)R_0(z_1)^{1/2}R_0(z_2)\lb{5.37}\\
&\quad=-R_0(z_1)^{1/2}[I_{\cH}-A_V(z_1)]^{-1}R_0(z_1)^{1/2}V[R(z_1)-R_0(z_2)]\no
\end{align}
and
\begin{align}
&(z_1-z_2)R_0(z_1)^{1/2}A_V(z_1)[I_{\cH}-A_V(z_1)]^{-1}R_0(z_1)^{1/2}\no\\
&\qquad\times R_0(z_2)^{1/2}A_V(z_2)[I_{\cH}-A_V(z_2)]^{-1}R_0(z_2)^{1/2}\no\\
&\quad = (z_1-z_2)R_0(z_1)^{1/2}[I_{\cH}-A_V(z_1)]^{-1}A_V(z_1)R_0(z_1)^{1/2}\no\\
&\qquad\times R_0(z_2)^{1/2}A_V(z_2)[I_{\cH}-A_V(z_2)]^{-1}R_0(z_2)^{1/2}\no\\
&\quad = R_0(z_1)^{1/2}[I_{\cH}-A_V(z_1)]^{-1}R_0(z_1)^{1/2}V[R_0(z_1)-R_0(z_2)]\no\\
&\qquad\times VR_0(z_2)^{1/2}[I_{\cH}-A_V(z_2)]^{-1}R_0(z_2)^{1/2}.\lb{5.38}
\end{align}
Distributing the factors to the resolvent differences $R_0(z_1)-R_0(z_2)$ in each of \eqref{5.36}, \eqref{5.37}, and \eqref{5.38}, and using the result \eqref{5.35}, one obtains
\begin{align}
&R(z_1)-R(z_2) - (z_1-z_2)R(z_1)R(z_2)    \no \\
&\quad=-R_0(z_1)^{1/2}[I_{\cH}-A_V(z_1)]^{-1}R_0(z_1)^{1/2}VR_0(z_2)^{1/2}\no\\
&\qquad\quad \times\big\{A_V(z_2)[I_{\cH}-A_V(z_2)]^{-1} + I_{\cH} \big\}R_0(z_2)^{1/2}\no\\
&\qquad + R_0(z_1)^{1/2}\big\{I_{\cH}+[I_{\cH}-A_V(z_1)]^{-1}A_V(z_1) \big\}\no\\
&\qquad\quad\times R_0(z_1)^{1/2}VR_0(z_2)^{1/2}[I_{\cH}-A_V(z_2)]^{-1}R_0(z_2)^{1/2}\no\\
&\quad= -R_0(z_1)^{1/2}[I_{\cH}-A_V(z_1)]^{-1}R_0(z_1)^{1/2}VR_0(z_2)^{1/2}[I_{\cH}-A_V(z_2)]^{-1}R_0(z_2)^{1/2}\no\\
&\qquad + R_0(z_1)^{1/2}[I_{\cH}-A_V(z_1)]^{-1}R_0(z_1)^{1/2}VR_0(z_2)^{1/2}[I_{\cH}-A_V(z_2)]^{-1}R_0(z_2)^{1/2}\no\\
&\quad=0.     \lb{5.39}
\end{align}
The second-to-last equality in \eqref{5.39} uses \eqref{5.29aa} twice.  
\end{proof}

\begin{remark}\lb{r2.10} 
$(i)$ In connection with Theorem \ref{t5.9} we also refer to the closely related result by Neiman-zade and Shkalikov  \cite[Thm.~4]{NS99}. \\[1mm] 
$(ii)$ Choosing $z\in \rho(H_0)\backslash \cD_{H_0,V}$, \eqref{5.26} shows that
\begin{equation}
\begin{split}
\dom(H)=\ran(R(z))\subseteq \ran\big((H_0-zI_{\cH})^{-1/2}\big)&=\dom\big((H_0-zI_{\cH})^{1/2}\big)\\
&=\dom\big(H_0^{1/2}\big)=\cH_{+1}(H_0).
\end{split}
\end{equation}
In particular, if $f\in \dom(H)$, then $Vf\in \cH_{-1}(H_0)$ is well-defined.\\[1mm]
$(iii)$  In view of \eqref{5.28}, purely formal manipulations in \eqref{5.26} lead to ``$H=H_0+V$.''  Thus, \eqref{5.28} may be viewed as a rigorous definition of the ``sum'' of $H_0$ and $V$. \hfill$\diamond$
\end{remark}

The following result, and the principal abstract result of this paper, is a Birman--Schwinger principle for the operator $H$ stated in terms of the operator $A_V(\dott)$, which functions as the corresponding Birman--Schwinger operator.

\begin{theorem}\lb{t2.11c}
Assume Hypothesis \ref{h5.4} and let $\lambda_0\in \rho(H_0)$.  If $H$ is the densely defined, closed operator in $\cH$ defined by \eqref{5.28}, then
\begin{equation}\lb{2.40u}
\text{$Hf = \lambda_0f$ for some $0\neq f \in \dom(H)$ implies $A_V(\lambda_0)g=g$,}
\end{equation}
where
\begin{equation}\lb{2.41u}
0\neq g=(H_0-\lambda_0I_{\cH})^{1/2}f.
\end{equation}
Conversely,
\begin{equation}
\text{$A_V(\lambda_0)g = g$ for some $0\neq g\in \cH$ implies $Hf=\lambda_0f$},
\end{equation}
where
\begin{equation}
0\neq f=R_0(\lambda_0)^{1/2}g\in \dom(H).
\end{equation}
Moreover,
\begin{equation}\lb{2.51u}
\dim(\ker(H-\lambda_0I_{\cH})) = \dim(\ker(I_{\cH}-A_V(\lambda_0)))<\infty,
\end{equation}
that is, the geometric multiplicity of $\lambda_0$ as an eigenvalue of $H$ coincides with the geometric multiplicity of the eigenvalue $1$ of $A_V(\lambda_0)$. 
In particular, for $z\in \rho(H_0)$,
\begin{equation}
\text{$z\in \rho(H)$ if and only if\, $1\in \rho(A_V(z))$}.
\end{equation}
\end{theorem}
\begin{proof}
By continuity of $\|A_V(z)\|_{\cB(\cH)}$ with respect to $z\in \rho(H_0)$, \eqref{5.23} implies the existence of $E_1<E_2<0$ such that $\|A_V(E)\|_{\cB(\cH)}<1$ for all $E\in (E_1,E_2)$.  In particular, $(E_1,E_2)\subset \rho(H_0)\backslash\cD_{H_0,V}$.

Let $\lambda_0\in \rho(H_0)$ and assume that $Hf=\lambda_0f$ for some $f\in \dom(H)\backslash\{0\}$.  By \eqref{5.26},
\begin{equation}
R(E)=R_0(E)^{1/2}[I_{\cH}-A_V(E)]^{-1}R_0(E)^{1/2},\quad E\in (E_1,E_2).
\end{equation}
Therefore, one infers (cf.~Remark \ref{r2.10}):
\begin{equation}\lb{2.47u}
[I_{\cH}-A_V(E)](H_0-EI_{\cH})^{1/2}R(E)=R_0(E)^{1/2},\quad E\in (E_1,E_2),
\end{equation}
which, upon applying both sides of \eqref{2.47u} to $(H-EI_{\cH})h$, $h\in \dom(H)$, yields:
\begin{equation}\lb{2.48u}
\begin{split}
[I_{\cH}-A_V(E)](H_0-EI_{\cH})^{1/2}h=R_0(E)^{1/2}(H-EI_{\cH})h,&\\
h\in \dom(H), \; E\in (E_1,E_2).&
\end{split}
\end{equation}
For fixed $h\in \dom(H)$, both sides of the equation in \eqref{2.48u} are $\cH$-valued analytic functions of $E\in \bbC\backslash [0,\infty)$ and are continuous in $\bbC_+$ up to points in $[0,\infty)\cap\rho(H_0)$.  Thus, for each fixed $h\in \dom(H)$, the equality in \eqref{2.48u} extends to all $E\in [\bbC\backslash[0,\infty)]\cup \rho(H_0)$.  In particular, \eqref{2.48u} with the choices $E=\lambda_0$ and $h=f$ implies
\begin{align}
[I_{\cH}-A_V(\lambda_0)](H_0-\lambda_0I_{\cH})^{1/2}f=R_0(\lambda_0)^{1/2}(H-\lambda_0I_{\cH})f=0,
\end{align}
which yields \eqref{2.40u} and \eqref{2.41u} (note that $g\neq 0$; otherwise, $f=R_0(\lambda_0)^{1/2}g=0$, contradicting the assumption on $f$).

To establish the converse, let $\lambda_0\in \rho(H_0)$ and assume that $A_V(\lambda_0)g=g$ for some $g\in \cH\backslash\{0\}$.  By \eqref{5.26},
\begin{align}\lb{2.53e}
I_{\cH}=(H-EI_{\cH})R_0(E)^{1/2}[I_{\cH}-A_V(E)]^{-1}R_0(E)^{1/2},\quad E\in (E_1,E_2).
\end{align}
Applying both sides of \eqref{2.53e} to $(H_0-EI_{\cH})^{1/2}[I_{\cH}-A_V(E)]h$, with $h\in \dom\big(H_0^{1/2}\big)$, $E\in (E_1,E_2)$, one obtains
\begin{equation}\lb{2.54e}
\begin{split}
(H_0-EI_{\cH})^{1/2}[I_{\cH}-A_V(E)]h=(H-EI_{\cH})R_0(E)^{1/2}h,&\\
E\in (E_1,E_2), \; h\in \dom\big(H_0^{1/2}\big).&
\end{split}
\end{equation}
For each fixed $h\in \dom\big(H_0^{1/2}\big)$, both sides of the equation in \eqref{2.54e} are $\cH$-valued analytic functions of $E\in \bbC\backslash [0,\infty)$ and are continuous in $\bbC_+$ up to points in $[0,\infty)\cap \rho(H_0)$.  Thus, for each fixed $h\in \dom\big(H_0^{1/2}\big)$, the equality in \eqref{2.54e} extends to all $E\in \bbC\backslash \sigma(H_0)$.  In particular, \eqref{2.54e} with the choices $E=\lambda_0$ and $h=g=A_V(\lambda_0)g=-R_0(\lambda_0)^{1/2}VR_0(\lambda_0)^{1/2}g\in \dom\big(H_0^{1/2}\big)$ implies
\begin{equation}
(H-\lambda_0I_{\cH})R_0(\lambda_0)^{1/2}g=(H_0-\lambda_0I_{\cH})^{1/2}\underbrace{[I_{\cH}-A_V(\lambda_0)]g}_{=0}=0.
\end{equation}
Thus, $(H-\lambda_0I_{\cH})f=0$ where $f=R_0(\lambda_0)^{1/2}g\neq 0$ since $g\neq 0$.

The equality in \eqref{2.51u} follows from the observation that $(H_0-\lambda_0I_{\cH})^{1/2}$ is a bijection, hence a vector space isomorphism, from $\ker(H-\lambda_0I_{\cH})$ to $\ker(I_{\cH}-A_V(\lambda_0))$.  Finally, $\dim(\ker(I_{\cH}-A_V(\lambda_0)))<\infty$ since $A_V(\lambda_0)$ is compact.
\end{proof}

\begin{remark} \lb{r2.12}
$(i)$ Theorem \ref{t5.9} implies that $\lambda_0$ in \eqref{2.40u} is necessarily an isolated point of $\sigma(H)$.\\[1mm]
\noindent
$(ii)$ One might wonder if also the algebraic multiplicity of $\lambda_0$ as an eigenvalue of $H$ coincides with the algebraic multiplicity of the eigenvalue $1$ of $A_V(\lambda_0)$. The answer is an emphatic no as the following elementary two-dimensional counterexample shows: Consider in $\cH = \bbC^2$, the operators 
\begin{align}
\begin{split} 
& H_0 = \begin{pmatrix} 1 & 0 \\ 0 & a \end{pmatrix}, \; a \in \bbR\backslash\{1\}, \quad 
V = \begin{pmatrix} 0 & 1 \\ -1 & -1 - a \end{pmatrix}, \\
& H = H_0 + V = \begin{pmatrix} 1 & 1 \\ -1 & -1 \end{pmatrix}, \quad H^2 = 0, 
\end{split} 
\end{align}
and 
\begin{align}
A_V(z) &= - (H_0 -z I_2)^{-1/2} V (H_0 -z I_2)^{-1/2}      \lb{2.59} \\
&=  \begin{pmatrix} 0 & - (a-z)^{-1/2} (1 - z)^{-1/2} \\  (a-z)^{-1/2} (1 - z)^{-1/2} & (1+a)(a-z)^{-1} \end{pmatrix}, \quad z \in \bbC \backslash \{1,a\},    \no \\
A_V(0) &=  \begin{pmatrix} 0 & - a^{-1/2} \\  a^{-1/2} & (1+a)a^{-1} \end{pmatrix}.     \lb{2.60}
\end{align}
(The precise choice of square root branch chosen in \eqref{2.59} and \eqref{2.60} is immaterial as long as the branch is consistently used.)
Then 
\begin{align}
\begin{split}
& {\det}_{\bbC^2} (H - z I_2) = z^2, \quad z \in \bbC,    \\
& {\det}_{\bbC^2} (A_V(0) - \zeta I_2) = (\zeta-1) \big(\zeta-a^{-1}\big), \quad \zeta \in \bbC,
\end{split} 
\end{align}
and hence with $m_g(T; z_0)$ and $m_a(T; z_0)$ denoting the geometric, respectively, algebraic multiplicity of a discrete eigenvalue $z_0 \in \bbC$ of the linear operator $T$ in some complex, separable Hilbert space, one obtains
\begin{equation}
m_g(A_V(0);1) = m_a(A_V(0);1) = m_g(H;0) = 1, \quad  m_a(H;0) = 2.
\end{equation}
\hfill$\diamond$
\end{remark} 

\section{Application:  One-Dimensional Schr\"odinger Operators \\ with $H^{-1}(\bbR)$ Potentials} \lb{s4}

If $V\in H^{-1}(\bbR)$, then $V$ is a Sobolev multiplier from $H^1(\bbR)$ to $H^{-1}(\bbR)$.  (For a precise definition of Sobolev multiplier, see Definition \ref{def_multiplier}.  A proof of the Sobolev multiplier property of distributions $V\in H^{-1}(\bbR)$ is provided, for completeness, in Proposition \ref{p2.7}.)  Hence,
\begin{equation}\lb{2.33}
V\in \cB\big(H^1(\bbR),H^{-1}(\bbR)\big)
\end{equation}
and \eqref{2.35a} holds.  In addition, Propositions \ref{p2.1} and \ref{p2.2} imply:
\begin{equation}\lb{2.34}
\begin{split}
&\big(H_0 + \kappa^2I_{L^2(\bbR)}\big)^{-1/2} \in \cB\big(H^{-1}(\bbR),L^2(\bbR)\big),\\
&\big(H_0 + \kappa^2I_{L^2(\bbR)}\big)^{-1/2} \in \cB\big(L^2(\bbR),H^1(\bbR)\big),\quad \kappa\in (0,\infty),
\end{split}
\end{equation}
where $H_0$ is defined according to (cf.~\eqref{3.12c})
\begin{equation}\lb{4.11c}
H_0f = -f'',\quad f\in \dom(H_0)=H^2(\bbR).
\end{equation}
In consequence, \eqref{2.33} and \eqref{2.34} combine to yield:
\begin{equation}\lb{2.41aa}
\big(H_0 + \kappa^2I_{L^2(\bbR)}\big)^{-1/2}V\big(H_0 + \kappa^2I_{L^2(\bbR)}\big)^{-1/2} \in \cB\big(L^2(\bbR)\big),\quad \kappa\in (0,\infty),
\end{equation}
with
\begin{align}
&\Big\|\big(H_0 + \kappa^2I_{L^2(\bbR)}\big)^{-1/2}V\big(H_0 + \kappa^2I_{L^2(\bbR)}\big)^{-1/2}\Big\|_{\cB(L^2(\bbR))}\lb{2.38a}\\
&\quad\leq M(\kappa)\|V\|_{\cB(H^1(\bbR),H^{-1}(\bbR))},\quad \kappa\in (0,\infty),\no
\end{align}
and
\begin{align}
M(\kappa) :&= \Big\|\big(H_0 + \kappa^2I_{L^2(\bbR)}\big)^{-1/2}\Big\|_{\cB(H^{-1}(\bbR),L^2(\bbR))}\no\\
&\quad \times\Big\|\big(H_0 + \kappa^2I_{L^2(\bbR)}\big)^{-1/2}\Big\|_{\cB(L^2(\bbR),H^1(\bbR))},\quad \kappa\in (0,\infty).\lb{2.43a}
\end{align}

As a map from $L^2(\bbR)$ to $L^2(\bbR)$, the free resolvent $\big(H_0+\kappa^2I_{L^2(\bbR)}\big)^{-1}$, $\kappa\in (0,\infty)$, is an integral operator with integral kernel given by the convolution kernel
\begin{equation}\lb{2.37}
G_0(-\kappa^2,x,y) = \frac{1}{2\kappa}e^{-\kappa|x-y|},\quad x,y\in \bbR,\quad \kappa\in (0,\infty).
\end{equation}
In addition, the square root $\big(H_0+\kappa^2I_{L^2(\bbR)}\big)^{-1/2}$ is also an integral operator with a convolution (integral) kernel given by
\begin{equation}
G_0^{1/2}(-\kappa^2,x,y) = \pi^{-1} H_0^{(1)}(i\kappa|x-y|),\quad x,y\in \bbR,\quad \kappa\in (0,\infty), 
\end{equation}
with $H_0^{(1)}(\dott)$ denoting the Hankel function of the first kind and order zero (cf.\ \cite[Sect.~9.1]{AS72}).

The result in \eqref{2.41aa} may be considerably strengthened as follows.

\begin{proposition}\lb{p2.9}
If $V\in H^{-1}(\bbR)$, then
\begin{equation}\lb{2.39}
\big(H_0 + \kappa^2I_{L^2(\bbR)}\big)^{-1/2}V\big(H_0 + \kappa^2I_{L^2(\bbR)}\big)^{-1/2} \in \cB_2\big(L^2(\bbR)\big),\quad \kappa\in (0,\infty),
\end{equation}
and
\begin{equation}\lb{2.47aa}
\begin{split}
\Big\| \big(H_0 + \kappa^2I_{L^2(\bbR)}\big)^{-1/2}V\big(H_0 + \kappa^2I_{L^2(\bbR)}\big)^{-1/2} \Big\|_{\cB_2(L^2(\bbR))}^2=\frac{1}{\kappa} \int_{\bbR} \frac{\big|\widehat{V}(\xi)\big|^2}{\xi^2+4\kappa^2}\, d\xi,&\\
\kappa\in (0,\infty).&
\end{split}
\end{equation}
In particular,
\begin{equation}\lb{3.18x}
\lim_{\kappa\to \infty} \Big\| \big(H_0 + \kappa^2I_{L^2(\bbR)}\big)^{-1/2}V\big(H_0 + \kappa^2I_{L^2(\bbR)}\big)^{-1/2} \Big\|_{\cB_2(L^2(\bbR))} = 0.
\end{equation}
Moreover, by \eqref{2.47aa}, the Hilbert--Schmidt property in \eqref{2.39} holds if and only if $V\in H^{-1}(\bbR)$.
\end{proposition}
\begin{proof}
Let $V\in H^{-1}(\bbR)$ and $\kappa\in (0,\infty)$.  We begin by proving \eqref{2.39} under the stronger assumption that $V\in \cS(\bbR)$.  Assuming that $V\in \cS(\bbR)$ and factoring $V$ according to
\begin{equation}\lb{3.18z}
V=uv,\quad v=|V|^{1/2},\quad u=|V|^{1/2}e^{i\arg(V)},
\end{equation}
one infers that $u,v,(|\dott|^2+\kappa^2)^{-1/2}\in L^2(\bbR)$.  Hence, \cite[Ch.~4]{Si05} implies
\begin{equation}\lb{3.19z}
v\big(H_0 + \kappa^2I_{L^2(\bbR)}\big)^{-1/2}, \, \big(H_0 + \kappa^2I_{L^2(\bbR)}\big)^{-1/2}u\in \cB_2\big(L^2(\bbR)\big),
\end{equation}
and as a consequence,
\begin{equation}\lb{2.42}
\big(H_0 + \kappa^2I_{L^2(\bbR)}\big)^{-1/2}V\big(H_0 + \kappa^2I_{L^2(\bbR)}\big)^{-1/2}\in \cB_1\big(L^2(\bbR)\big) \subset \cB_2\big(L^2(\bbR)\big).
\end{equation}
Furthermore,
\begin{align}
&\Big\| \big(H_0 + \kappa^2I_{L^2(\bbR)}\big)^{-1/2}V\big(H_0 + \kappa^2I_{L^2(\bbR)}\big)^{-1/2} \Big\|_{\cB_2(L^2(\bbR))}^2\no\\
&\quad =\tr_{L^2(\bbR)}\Big(\big(H_0 + \kappa^2I_{L^2(\bbR)}\big)^{-1/2}V^*\big(H_0 + \kappa^2I_{L^2(\bbR)}\big)^{-1}V\big(H_0 + \kappa^2I_{L^2(\bbR)}\big)^{-1/2}\Big)\no\\
&\quad =\tr_{L^2(\bbR)}\Big(V^*\big(H_0 + \kappa^2I_{L^2(\bbR)}\big)^{-1}V\big(H_0 + \kappa^2I_{L^2(\bbR)}\big)^{-1}\Big).\lb{2.43}
\end{align}
The final equality in \eqref{2.43} is an application of the cyclicity property of the trace.  Using \eqref{2.37}, the trace in \eqref{2.43} may be computed as follows:
\begin{align}
&\tr_{L^2(\bbR)}\Big(V^*\big(H_0 + \kappa^2I_{L^2(\bbR)}\big)^{-1}V\big(H_0 + \kappa^2I_{L^2(\bbR)}\big)^{-1}\Big)\no\\
&\quad= \frac{1}{\kappa} \Big(V,\big(H_0+4\kappa^2I_{L^2(\bbR)}\big)^{-1}V\Big)_{L^2(\bbR)}\no\\
&\quad= \frac{1}{\kappa} \int_{\bbR} \frac{\big|\widehat{V}(\xi)\big|^2}{\xi^2+4\kappa^2}\, d\xi,\lb{2.44}
\end{align}
where the final equality in \eqref{2.44} in an application of Plancherel's identity.  The result of \eqref{2.44} in \eqref{2.43} is
\begin{equation}\lb{2.45}
\Big\| \big(H_0 + \kappa^2I_{L^2(\bbR)}\big)^{-1/2}V\big(H_0 + \kappa^2I_{L^2(\bbR)}\big)^{-1/2} \Big\|_{\cB_2(L^2(\bbR))}^2= \frac{1}{\kappa} \int_{\bbR} \frac{\big|\widehat{V}(\xi)\big|^2}{\xi^2+4\kappa^2}\, d\xi.
\end{equation}

Returning to arbitrary $V\in H^{-1}(\bbR)$, the idea is to now approximate $V$ by functions in $\cS(\bbR)$ and appeal to \eqref{2.45}.  Using the fact that $\cS(\bbR)$ is dense in $H^{-1}(\bbR)$ (see, e.g., \cite[Proposition 3.17]{HT08}), let $\{V_n\}_{n=1}^{\infty}\subset\cS(\bbR)$ be a sequence with the property that $\|V_n-V\|_{H^{-1}(\bbR)}\to 0$ as $n\to \infty$.  By \eqref{2.45},
\begin{align}
&\Big\| \big(H_0 + \kappa^2I_{L^2(\bbR)}\big)^{-1/2}(V_n-V_m)\big(H_0 + \kappa^2I_{L^2(\bbR)}\big)^{-1/2} \Big\|_{\cB_2(L^2(\bbR))}^2\no\\
&\quad= \frac{1}{\kappa}\int_{\bbR}\frac{\big|\widehat{V}_n(\xi)-\widehat{V}_m(\xi)\big|^2}{\xi^2+4\kappa^2}\, d\xi\no\\
&\quad\leq \frac{c(\kappa)}{\kappa} \|V_n-V_m\|_{H^{-1}(\bbR)}^2,\quad m,n\in \bbN,\lb{2.46}
\end{align}
where $c(\kappa):=\max\{1,1/(4\kappa^2)\}$.  In particular, the estimate in \eqref{2.46} implies that the sequence
\begin{equation}\lb{2.54}
\Big\{\big(H_0 + \kappa^2I_{L^2(\bbR)}\big)^{-1/2}V_n\big(H_0 + \kappa^2I_{L^2(\bbR)}\big)^{-1/2}\Big\}_{n\in\bbN}
\end{equation}
is Cauchy, thus convergent, in $\cB_2\big(L^2(\bbR)\big)$.  Let $-A_V(-\kappa^2)\in \cB_2\big(L^2(\bbR)\big)$ denote the limit of the sequence in \eqref{2.54} in $\cB_2\big(L^2(\bbR)\big)$.  (The leading negative sign in $-A_V(-\kappa^2)$ is taken to be consistent with the sign convention in \eqref{5.22}.)  On the other hand, by \eqref{2.35a},
\begin{align}
&\Big\| \big(H_0 + \kappa^2I_{L^2(\bbR)}\big)^{-1/2}(V_n-V)\big(H_0 + \kappa^2I_{L^2(\bbR)}\big)^{-1/2} \Big\|_{\cB(L^2(\bbR))}\no\\
&\quad \leq 2^{1/2}M(\kappa)\|V_n-V\|_{H^{-1}(\bbR)},\quad n\in \bbN,\lb{2.55}
\end{align}
where $M(\kappa)$ is the constant in \eqref{2.43a}.  The estimate in \eqref{2.55} shows that \eqref{2.54} converges in $\cB\big(L^2(\bbR)\big)$-norm to
\begin{equation}
\big(H_0 + \kappa^2I_{L^2(\bbR)}\big)^{-1/2}V\big(H_0 + \kappa^2I_{L^2(\bbR)}\big)^{-1/2}.
\end{equation}
Since convergence in $\cB_2\big(L^2(\bbR)\big)$ implies convergence in $\cB\big(L^2(\bbR)\big)$, and limits in $\cB\big(L^2(\bbR)\big)$ are unique, it follows that
\begin{equation}
\big(H_0 + \kappa^2I_{L^2(\bbR)}\big)^{-1/2}V\big(H_0 + \kappa^2I_{L^2(\bbR)}\big)^{-1/2}=-A_V(-\kappa^2)\in \cB_2\big(L^2(\bbR)\big).
\end{equation}
Finally, \eqref{2.45} implies
\begin{equation}\lb{4.31}
\begin{split}
&\Big\| \big(H_0 + \kappa^2I_{L^2(\bbR)}\big)^{-1/2}V_n\big(H_0 + \kappa^2I_{L^2(\bbR)}\big)^{-1/2} \Big\|_{\cB_2(L^2(\bbR))}^2\\
&\quad = \frac{1}{\kappa} \int_{\bbR} \frac{\big|\widehat{V}_n(\xi)\big|^2}{\xi^2+4\kappa^2}\, d\xi,\quad n\in \bbN.
\end{split}
\end{equation}
Taking $n\to \infty$ throughout \eqref{4.31} yields \eqref{2.47aa}.

To prove \eqref{3.18x}, note that for $\kappa\geq 1/2$,
\begin{equation}
0\leq \frac{\xi^2+1}{\xi^2+4\kappa^2} \leq 1,\quad \xi\in \bbR,
\end{equation}
so that
\begin{equation}
\int_{\bbR}\frac{\big|\widehat{V}(\xi) \big|^2}{\xi^2+4\kappa^2}\, d\xi \leq \|V\|_{H^{-1}(\bbR)}^2, \quad \kappa\in [1/2,\infty).
\end{equation}
Thus, \eqref{3.18x} follows by taking $\kappa\to \infty$ in \eqref{2.47aa}.
\end{proof}

\begin{remark}\lb{r4.5c}
$(i)$  The results of Proposition \ref{p2.9} are known.  In fact, both \eqref{2.39} and \eqref{2.47aa} appear in \cite[Prop.~2.1]{KVZ18} and the lecture notes \cite[Lemma 4.3.1]{Ko} (see also \cite[Prop.~2.1]{KV19}, \cite[eq.~(1.11)]{CKV22}). 

\vspace{2mm}
\noindent
$(ii)$ If $z\in \bbC\backslash[0,\infty)$ and $V\in H^{-1}(\bbR)$, then \eqref{2.39} and \eqref{2.18aa} imply
\begin{align}
& \big(H_0-zI_{L^2(\bbR)}\big)^{-1/2} V \big(H_0-zI_{L^2(\bbR)}\big)^{-1/2}\no\\
&\quad = \overline{\big(H_0-zI_{L^2(\bbR)}\big)^{-1/2} \big(H_0+I_{L^2(\bbR)}\big)^{1/2}} 
\big(H_0+I_{L^2(\bbR)}\big)^{-1/2} V \big(H_0+I_{L^2(\bbR)}\big)^{-1/2}\no\\
&\qquad \times \big(H_0+I_{L^2(\bbR)}\big)^{1/2} \big(H_0-zI_{L^2(\bbR)}\big)^{-1/2}\lb{2.61}
\end{align}
with
\begin{align}
&\overline{\big(H_0-zI_{L^2(\bbR)}\big)^{-1/2}\big(H_0+I_{L^2(\bbR)}\big)^{1/2}}\in \cB\big(L^2(\bbR)\big),\no\\
&\big(H_0+I_{L^2(\bbR)}\big)^{1/2}\big(H_0-zI_{L^2(\bbR)}\big)^{-1/2}\in \cB\big(L^2(\bbR)\big),\\
&\big(H_0+I_{L^2(\bbR)}\big)^{-1/2}V\big(H_0+I_{L^2(\bbR)}\big)^{-1/2}\in \cB_2\big(L^2(\bbR)\big).\no
\end{align}
Hence, for $V\in H^{-1}(\bbR)$,
\begin{equation}\lb{2.63}
\big(H_0-zI_{L^2(\bbR)}\big)^{-1/2} V \big(H_0-zI_{L^2(\bbR)}\big)^{-1/2}\in \cB_2\big(L^2(\bbR)\big),\quad z\in \bbC\backslash[0,\infty).
\end{equation}
\hfill$\diamond$
\end{remark}

By Proposition \ref{p2.9}, one infers that Hypothesis \ref{h5.4} holds with $H_0$ defined by \eqref{4.11c}, a fixed $V\in H^{-1}(\bbR)$, $\cH_{+1}(H_0)=H^1(\bbR)$, and $\cH_{-1}(H_0)=H^{-1}(\bbR)$.  Therefore, invoking Theorem \ref{t5.9}, one may define a (densely defined) $L^2(\bbR)$ realization $H$ of the purely formal sum ``$H_0+V$'' indirectly according to \eqref{5.22} and \eqref{5.28} as follows.  Set
\begin{equation}\lb{4.37c}
A_V(z):= - \big(H_0-zI_{L^2(\bbR)}\big)^{-1/2} V \big(H_0-zI_{L^2(\bbR)}\big)^{-1/2},\quad z\in \bbC\backslash[0,\infty),
\end{equation}
and
\begin{equation}\lb{4.38c}
\begin{split}
R(z) := \big(H_0 - zI_{L^2(\bbR)}\big)^{-1/2} \big[I_{L^2(\bbR)}-A_V(z)\big]^{-1} \big(H_0 - zI_{L^2(\bbR)}\big)^{-1/2},&\\
z\in \{\zeta\in \rho(H_0)\,|\, 1\in \rho(A_V(\zeta))\}=\bbC\backslash([0,\infty)\cup\cD_{H_0,V}),&
\end{split}
\end{equation}
where $\cD_{H_0,V}$ is the discrete set guaranteed to exist by Proposition \ref{p5.6}. A combination of Theorems \ref{t5.9} and \ref{t2.11c} together with \eqref{5.29a} and \eqref{2.39} then yield the following result.

\begin{theorem}\lb{t4.6c}
Let $V\in H^{-1}(\bbR)$ and let $A_V(\dott)$ and $R(\dott)$ be defined by \eqref{4.37c} and \eqref{4.38c}.  The map $R(\dott)$ uniquely defines a densely defined, closed, linear operator $H$ in $L^2(\bbR)$ by
\begin{equation}\lb{5.28a}
R(z) = \big(H-zI_{L^2(\bbR)}\big)^{-1},\quad z\in \bbC\backslash([0,\infty)\cup \cD_{H_0,V}).
\end{equation}
In addition, $(\bbC\backslash([0,\infty)\cup \cD_{H_0,V}))\subseteq \rho(H)$ and $H$ has the property:
\begin{equation}\lb{5.29b}
\Big[\big(H-zI_{L^2(\bbR)}\big)^{-1} - \big(H_0-zI_{L^2(\bbR)}\big)^{-1}\Big] \in \cB_2(\cH),\quad z\in \rho(H_0)\cap\rho(H).
\end{equation}
Moreover, a point $\lambda_0\in \bbC\backslash[0,\infty)$, is an eigenvalue of $H$ with geometric multiplicity $k\in \bbN$ if and only if $1$ is an eigenvalue of $A_V(\lambda_0)$ with geometric multiplicity $k$.
\end{theorem}

In the proof of Proposition \ref{p2.9}, the norm identity in \eqref{2.47aa} is computed by expressing the square of the norm as a trace, see \eqref{2.43}.  The trace is tractable as it may be calculated using the Green's function of the resolvent of $H_0$; this is the calculation carried out in \eqref{2.44}.  Alternatively, the trace may be calculated using integral kernels in Fourier space.  As an illustration of the various possible approaches in connection with calculating the trace in \eqref{2.43}, we present the details of the Fourier transform approach next.  This approach relies on the following integration result.

\begin{proposition}\lb{p3.6}
If $\eta\in \bbR$ and $\kappa\in (0,\infty)$, then
\begin{equation}\lb{3.33}
\frac{\kappa}{2\pi}\int_{\bbR} \frac{d\xi}{\big[\xi^2+\kappa^2\big]\big[(\eta-\xi)^2+\kappa^2\big]} = \frac{1}{\eta^2+4\kappa^2}.
\end{equation}
\end{proposition}
\begin{proof}
Let $I_{\kappa}(\eta)$, $\eta\in \bbR$, denote the integral on the left-hand side of the equality in \eqref{3.33}.  The function $I_{\kappa}(\dott)$ is a convolution:
\begin{equation}\lb{3.34}
I_{\kappa}(\eta) = \bigg[\frac{1}{|\dott|^2+\kappa^2} \ast \frac{1}{|\dott|^2+\kappa^2} \bigg](\eta),\quad \eta\in \bbR.
\end{equation}
By the convolution theorem expressed in the form $f \ast g = (2\pi)^{1/2} \cF \big[\cF^{-1}f\cdot\cF^{-1}g\big]$, \eqref{3.34} may be recast as
\begin{equation}\lb{3.35}
I_{\kappa}(\dott) = (2\pi)^{1/2} \cF\Bigg\{\cF^{-1}\bigg\{ \frac{1}{|\dott|^2+\kappa^2} \bigg\} \cdot \cF^{-1}\bigg\{ \frac{1}{|\dott|^2+\kappa^2} \bigg\} \Bigg\}.
\end{equation}
One computes:
\begin{align}
\cF^{-1}\bigg\{ \frac{1}{|\dott|^2+\kappa^2} \bigg\}(x) 
= \bigg(\frac{2}{\pi}\bigg)^{1/2} \int_0^{\infty} \frac{\cos(|x|\xi)}{\xi^2+\kappa^2}\, d\xi 
= \bigg(\frac{\pi}{2}\bigg)^{1/2} \frac{e^{-\kappa|x|}}{\kappa},\lb{3.36}
\end{align}
where the final equality in \eqref{3.36} follows from an application of \cite[3.723.2]{GR80}.  Thus, by \eqref{3.35} and \eqref{3.36},
\begin{align}\lb{3.37}
I_{\kappa}(\eta) &= \frac{(2\pi)^{3/2}}{4\kappa^2}\cF\big\{e^{-2\kappa|\dott|} \big\}(\eta) = \frac{2\pi}{\kappa} \frac{1}{\eta^2+4\kappa^2},\quad \eta\in \bbR.
\end{align}
The Fourier transform in \eqref{3.37} may be deduced from \eqref{3.36} or, alternatively, directly computed as follows:
\begin{align}
\cF\big\{e^{-2\kappa|\dott|} \big\}(\eta) 
= \frac{1}{(2\pi)^{1/2}}\cdot 2 \int_0^{\infty} \cos(|\eta|x)e^{-2\kappa x}\, dx 
= \frac{1}{(2\pi)^{1/2}} \frac{4\kappa}{\eta^2+4\kappa^2},\quad \eta\in \bbR.\lb{3.38}
\end{align}
The final equality in \eqref{3.38} is an application of \cite[3.893.2]{GR80}.  Finally, \eqref{3.33} immediately follows from \eqref{3.37} after some minor algebraic manipulations.
\end{proof}

\begin{proposition}\lb{p3.7}
If $V\in \cS(\bbR)$ and $\kappa\in (0,\infty)$, then
\begin{equation}\lb{3.39}
A_V(-\kappa^2) := -\big(H_0+\kappa^2I_{L^2(\bbR)}\big)^{-1/2}V\big(H_0+\kappa^2I_{L^2(\bbR)}\big)^{-1/2} \in \cB_1\big(L^2(\bbR)\big)
\end{equation}
and, recalling the factorization $V=uv$ introduced in \eqref{3.18z},
\begin{align}
\tr_{L^2(\bbR)}\big(A_V(-\kappa^2)\big) &= -\frac{1}{2\kappa}\int_{\bbR}V(x) \, dx\no\\
&= -\tr_{L^2(\bbR)}\Big(u\big(H_0+\kappa^2I_{L^2(\bbR)}\big)^{-1}v \Big),\lb{3.40}\\
\big\|A_V(-\kappa^2)\big\|_{\cB_2(L^2(\bbR))}^2 &= \frac{1}{\kappa}\int_{\bbR} \frac{\big|\widehat{V}(\eta)\big|^2}{\eta^2+4\kappa^2}\, d\eta.\lb{3.41}
\end{align}
\end{proposition}
\begin{proof}
The trace class containment in \eqref{3.39} immediately follows from
\begin{equation}
v\big(H_0 + \kappa^2I_{L^2(\bbR)}\big)^{-1/2},\big(H_0 + \kappa^2I_{L^2(\bbR)}\big)^{-1/2}u\in \cB_2\big(L^2(\bbR)\big),
\end{equation}
which is a consequence of \cite[Ch.~4]{Si05}.  In turn, to deduce \eqref{3.40}, one notes that $\cF A_V(-\kappa^2) \cF^{-1}$ is an integral operator with integral kernel
\begin{equation}\lb{3.43}
\begin{split}
\big[\cF A_V(\kappa) \cF^{-1}\big](\xi,\xi') = -(2\pi)^{-1/2} \big(\xi^2+\kappa^2\big)^{-1/2} \widehat{V}(\xi-\xi')\big((\xi')^2+\kappa^2\big)^{-1/2},&\\
\xi,\xi'\in \bbR.&
\end{split}
\end{equation}
In consequence,
\begin{align}
& \tr_{L^2(\bbR)}\big(A_V(-\kappa^2) \big) = \int_{\bbR} \big[\cF A_V(-\kappa^2) \cF^{-1}\big](\xi,\xi)\, d\xi
= -(2\pi)^{-1/2}\, \widehat{V}(0)\int_{\bbR} \frac{d\xi}{\xi^2+\kappa^2}   \no \\ 
& \quad = -\frac{1}{2\pi} \int_{\bbR}V(x)\, dx \cdot \int_{\bbR}\frac{d\xi}{\xi^2+\kappa^2} 
= -\frac{1}{2\kappa} \int_{\bbR} V(x)\, dx.
\end{align}
The second equality in \eqref{3.40} follows from the cyclicity property of the trace functional.

Turning to \eqref{3.41}, the operator
\begin{equation}
\cF A_V(-\kappa^2)^*A_V(-\kappa^2)\cF^{-1} = \big[\cF A_V(-\kappa^2)^*\cF^{-1}\big]\big[\cF A_V(-\kappa^2)\cF^{-1}\big]
\end{equation}
is an integral operator with integral kernel
\begin{align}
&\big[\cF A_V(-\kappa^2)^*A_V(-\kappa^2)\cF^{-1}\big](\xi,\xi')\\
&\quad = (2\pi)^{-1} \big(\xi^2+\kappa^2\big)^{-1/2}\no\\
&\qquad \times\int_{\bbR} \Big\{\overline{\widehat{V}(\xi''-\xi)} \big((\xi'')^2+\kappa^2\big)^{-1}\widehat{V}(\xi''-\xi')\big((\xi')^2+\kappa^2\big)^{-1/2}\Big\}\, d\xi'',\quad \xi,\xi'\in \bbR.\no
\end{align}
Hence, one calculates:
\begin{align}
& \big\|A_V(-\kappa^2)\big\|_{\cB_2(L^2(\bbR))}^2 = \tr_{L^2(\bbR)}\big(A_V(-\kappa^2)^*A_V(-\kappa^2) \big)\no\\
& \quad = \tr_{L^2(\bbR)}\big(\cF A_V(-\kappa^2)^*A_V(-\kappa^2)\cF^{-1} \big)\no\\
& \quad = \frac{1}{2\pi} \int_{\bbR}\Bigg\{ \big(\xi^2+\kappa^2\big)^{-1} \int_{\bbR} \frac{\big|\widehat{V}(\eta) \big|^2}{(\eta+\xi)^2+\kappa^2}\, d\eta \Bigg\}\, d\xi\no\\
& \quad = \frac{1}{2\pi} \int_{\bbR}\Bigg\{ \big|\widehat{V}(\eta)\big|^2 \int_{\bbR} \frac{d\xi}{\big[\xi^2+\kappa^2\big]\big[(\eta-\xi)^2+\kappa^2\big]}\Bigg\}\, d\eta 
= \frac{1}{\kappa} \int_{\bbR}\frac{\big|\widehat{V}(\eta) \big|^2}{\eta^2+\kappa^2}\, d\eta.\lb{3.46}
\end{align}
The final equality in \eqref{3.46} follows from Proposition \ref{p3.6}.
\end{proof}

\begin{remark} \lb{r4.9} 
The derivation of \eqref{3.41} given in the proof of Proposition \ref{p3.7} is based on expressing the (square of the) Hilbert--Schmidt norm in terms of a trace (see, \eqref{3.46}).  The trace is then computed in Fourier space by integrating over the diagonal of the corresponding integral kernel.  Alternatively, the following Lemma \ref{l2.9} may be used to compute the trace as follows:
\begin{align}
& \tr_{L^2(\bbR)}\big(A_V(-\kappa^2)^*A_V(-\kappa^2) \big) = \frac{1}{4\kappa^2} \int_{\bbR\times \bbR} \ol{V(x)} e^{-2\kappa|x-x'|}V(x')    \lb{3.48y} \\
& \quad = \frac{1}{4\kappa^2} (2\pi)^{1/2} \int_{\bbR} \big|\widehat{V}(\xi)\big|^2\frac{4\kappa}{4\kappa^2+\xi^2}(2\pi)^{-1/2}\, d\xi 
= \frac{1}{\kappa} \int_{\bbR} \frac{\big|\widehat{V}(\xi)\big|^2}{4\kappa^2+\xi^2} \, d\xi, 
\quad \kappa\in (0,\infty).    \no 
\end{align}
In \eqref{3.48y}, one uses the Fourier transform of $e^{-2\kappa|\dott|}$, $\kappa\in (0,\infty)$, given by \eqref{3.38}.\hfill$\diamond$
\end{remark}

The following chain of results, culminating with Lemma \ref{l2.9}, is used to justify the calculation in \eqref{3.48y}.  We present these results here for completeness.  We recall the convolution theorem for Schwartz functions in the following form (see, e.g., \cite[Theorem IX.3]{RS75}).

\begin{lemma} \lb{lA.2}
If $f,g\in \cS(\bbR)$, then
\begin{equation}
\cF(f\ast g)=(2\pi)^{1/2}(\cF f)\cdot (\cF g)\quad \text{and}\quad \cF^{-1}(f\ast g) = (2\pi)^{1/2} (\cF^{-1}f)\cdot(\cF^{-1}g).
\end{equation}
\end{lemma}

In addition, we recall that for $f\in L^1(\bbR)$, $\widehat{f}$ is continuous with $\lim_{|\xi|\to\infty}f(\xi)=0$ and
\begin{equation}\lb{A.2}
\big\| \widehat{f}\,\big\|_{L^{\infty}(\bbR)} \leq \|f\|_{L^1(\bbR)},\quad f\in L^1(\bbR).
\end{equation}
The following result is an extension of Lemma \ref{lA.2}.

\begin{theorem} \lb{tA.3}
The following statements hold.\\[1mm]
$(i)$ If $f,g\in L^2(\bbR)$, then $f\ast g\in L^{\infty}(\bbR)$ and
\begin{equation}\lb{A.7}
f\ast g = (2\pi)^{1/2} \cF\big\{(\cF^{-1}f)\cdot(\cF^{-1}g) \big\}.
\end{equation}
$(ii)$ If $f\in L^2(\bbR)$ and $g\in L^1(\bbR)$, then $f\ast g\in L^2(\bbR)$ and
\begin{equation}\lb{A.9z}
\widehat{f\ast g} = (2\pi)^{1/2} \widehat{f}\widehat{g}.
\end{equation}
\end{theorem}
\begin{proof}
To prove $(i)$, let $f,g\in L^2(\bbR)$.  Young's inequality with $p=q=2$ and $r=\infty$ implies $f\ast g\in L^{\infty}(\bbR)$.  It remains to prove \eqref{A.7}.  Since $\cF^{-1}f,\cF^{-1}g\in L^2(\bbR)$, H\"older's inequality implies $\big(\cF^{-1}f\big)\cdot\big(\cF^{-1}g \big)\in L^1(\bbR)$.  Thus, $\cF\big\{\big(\cF^{-1}f\big)\cdot\big(\cF^{-1}g \big)\big\}$ is well-defined and belongs to $L^{\infty}(\bbR)$.  To prove the equality in \eqref{A.7}, choose sequences $\{f_j\}_{j=1}^{\infty}, \{g_j\}_{j=1}^{\infty}\subset \cS(\bbR)$ such that
\begin{equation}
\lim_{j\to \infty}\|f_j-f\|_{L^2(\bbR)} = 0 = \lim_{j\to \infty} \|g_j-g\|_{L^2(\bbR)}.
\end{equation}
Applying the inequality in \eqref{A.2}, H\"older's inequality, and the unitary property of $\cF^{-1}$, one obtains
\begin{align}
&\big\|\cF\big\{\big(\cF^{-1}f_j\big)\cdot\big(\cF^{-1}g_j \big)\big\} - \cF\big\{\big(\cF^{-1}f\big)\cdot\big(\cF^{-1}g \big)\big\} \big\|_{L^{\infty}(\bbR)}\no\\
&\quad\leq \big\|\big(\cF^{-1}f_j\big)\cdot\big(\cF^{-1}g_j \big) - \big(\cF^{-1}f\big)\cdot\big(\cF^{-1}g \big) \big\|_{L^1(\bbR)}\no\\
&\quad\leq \|f_j - f\|_{L^2(\bbR)}\|g_j \|_{L^2(\bbR)} + \|f \|_{L^2(\bbR)}\|g_j - g\|_{L^2(\bbR)},\quad j\in \bbN.\lb{A.10}
\end{align}
In particular, \eqref{A.10} implies
\begin{equation}
\lim_{j\to \infty} \big\|\cF\big\{\big(\cF^{-1}f_j\big)\cdot\big(\cF^{-1}g_j \big)\big\} - \cF\big\{\big(\cF^{-1}f\big)\cdot\big(\cF^{-1}g \big)\big\} \big\|_{L^{\infty}(\bbR)} = 0,
\end{equation}
or equivalently, by Lemma \ref{lA.2},
\begin{equation}\lb{A.12}
\lim_{j\to \infty} \big\|(2\pi)^{-n/2}f_j\ast g_j - \cF\big\{\big(\cF^{-1}f\big)\cdot\big(\cF^{-1}g \big)\big\} \big\|_{L^{\infty}(\bbR)} = 0.
\end{equation}
On the other hand,
\begin{align}
&\|(f_j\ast g_j) - (f\ast g) \|_{L^{\infty}(\bbR)}\no\\
&\quad\leq \|(f_j-f)\ast g_j\|_{L^{\infty}(\bbR)} + \|f\ast (g_j - g) \|_{L^{\infty}(\bbR)}\no\\
&\quad\leq \|f_j-f\|_{L^2(\bbR)}\|g\|_{L^2(\bbR)} + \|f\|_{L^2(\bbR)}\|g_j-g\|_{L^2(\bbR)},\quad j\in \bbN,
\end{align}
which implies
\begin{equation}\lb{A.14}
\lim_{j\to \infty} \big\|(2\pi)^{-n/2}(f_j\ast g_j) - (2\pi)^{-n/2}(f\ast g) \big\|_{L^{\infty}(\bbR)}=0.
\end{equation}
By uniqueness of limits, \eqref{A.12} and \eqref{A.14} imply \eqref{A.7}.

The proof of $(ii)$ is entirely analogous to the proof of $(i)$, with only minor modifications required.  For $f\in L^2(\bbR)$ and $g\in L^1(\bbR)$, Young's inequality with $p=2$, $q=1$, and $r=2$ implies $f\ast g\in L^2(\bbR)$.  In particular $\cF(f\ast g)$ is well-defined and belongs to $L^2(\bbR)$.  One chooses $\{f_j\}_{j=1}^{\infty},\{g_j\}_{j=1}^{\infty}\subset \cS(\bbR)$ such that
\begin{equation}
\lim_{j\to \infty}\|f_j-f\|_{L^2(\bbR)} = 0 = \lim_{j\to \infty}\|g_j-g\|_{L^1(\bbR)}.
\end{equation}
Therefore, $f_j\ast g_j \in L^2(\bbR)$, $j\in \bbN$, and
\begin{align}
&\|\cF(f_j\ast g_j) - \cF(f\ast j) \|_{L^2(\bbR)}\no\\
&\quad=\|(f_j\ast g_j) - (f\ast j) \|_{L^2(\bbR)}\no\\
&\quad\leq \|(f_j-f)\ast g_j\|_{L^2(\bbR)} + \|f\ast (g_j-g)\|_{L^2(\bbR)}\lb{A.17}\\
&\quad\leq \|f_j-f\|_{L^2(\bbR)}\|g_j\|_{L^1(\bbR)} + \|f\|_{L^2(\bbR)}\|g_j-g\|_{L^1(\bbR)},\quad j\in \bbN.\no
\end{align}
In addition,
\begin{align}
&\big\|(\cF f_j)(\cF g_j)-(\cF f)(\cF g) \big\|_{L^2(\bbR)}     \no \\
&\quad\leq \big\|\big(\cF f_j-\cF f\big)(\cF g_j) \big\|_{L^2(\bbR)} + \big\|(\cF f)(\cF g_j-\cF g) \big\|_{L^2(\bbR)}\no\\
&\quad\leq \big\|f_j-f \big\|_{L^2(\bbR)}\|g_j \|_{L^1(\bbR)} + \| f\|_{L^2(\bbR)}\big\|g_j-g \big\|_{L^1(\bbR)},\quad j\in \bbN.    \lb{A.18} 
\end{align}
Theorem \ref{lA.2} implies $\cF(f_j\ast g_j) = (2\pi)^{1/2} (\cF f_j)\cdot (\cF g_j)$, $j\in \bbN$.  Therefore, \eqref{A.17} and \eqref{A.18}, combined with uniqueness of limits, yield \eqref{A.9z}.
\end{proof}

The following result is useful for calculating certain double integrals in Fourier space (cf., e.g., \cite[p.~13]{Si71}).

\begin{lemma}\lb{l2.9}
If $f,g,\in L^2(\bbR)$ and $h\in L^1(\bbR)$, then
\begin{equation}
\int_{\bbR\times \bbR} \ol{f(x)}h(x-x')g(x')\, dx\, dx' = (2\pi)^{1/2} \int_{\bbR}\ol{\widehat{f}(\xi)}\widehat{h}(\xi)\widehat{g}(\xi)\, d\xi.
\end{equation}
\end{lemma}
\begin{proof}
Under the hypothesis on $f$, $g$, and $h$, the convolution theorem (Theorem \ref{tA.3} $(ii)$) implies:
\begin{align}
\begin{split} 
& \int_{\bbR\times \bbR} \ol{f(x)}h(x-x')g(x')\, dx\, dx' = \int_{\bbR} \ol{f(x)}(h\ast g)(x)\, dx\\
& \quad = \int_{\bbR} \ol{\widehat{f}(\xi)}\widehat{(h\ast g)}(\xi)\, d\xi 
= (2\pi)^{1/2} \int_{\bbR}\ol{\widehat{f}(\xi)}\widehat{h}(\xi)\widehat{g}(\xi)\, d\xi.   
\end{split} 
\end{align}
\end{proof}

\appendix
\section{Some Basic Facts on Sobolev Multipliers} \lb{s0}

In this section, we recall some fundamental results on Sobolev multipliers. The Schwartz space of rapidly decreasing functions on $\bbR$ is denoted by $\cS(\bbR)$, and $\cS'(\bbR)$ denotes the corresponding set of tempered distributions.  The convention employed for the Fourier transform is:
\begin{equation}\lb{2.1}
\begin{split}
(\cF f)(\xi)\equiv \widehat f (\xi) = (2\pi)^{-1/2} \int_{\bbR} e^{-i\xi\cdot x} f(x)\, dx\, \text{ for a.e.~$\xi\in \bbR$},\\
\big(\cF^{-1} f\big)(x)\equiv \widecheck f (x) = (2\pi)^{-1/2} \int_{\bbR} e^{ix\cdot \xi} f(\xi)\, d\xi\, \text{ for a.e.~$x\in \bbR$}.
\end{split}
\end{equation}
In addition, the convolution is defined according to
\begin{equation}\lb{3.2c}
(f \ast g)(x) = \int_{\bbR} f(y)g(x-y)\, dy = \int_{\bbR} f(x-y)g(y)\, dy,
\end{equation}
and Young's inequality holds in the following form.

\begin{theorem}[Young's Inequality]
Let $1\leq p,q,r\leq \infty$ with $p^{-1}+q^{-1} = 1 + r^{-1}$.  If $f\in L^p(\bbR)$ and $g\in L^q(\bbR)$ then $f\ast g\in L^r(\bbR)$ and
\begin{equation}
\|f\ast g\|_{L^r(\bbR)} \leq \|f\|_{L^p(\bbR)}\|g\|_{L^q(\bbR)}.
\end{equation}
\end{theorem}

One notes that Young's inequality implies the following inclusions:
\begin{align}
L^1(\bbR)\ast L^1(\bbR)\subset L^1(\bbR),  \quad 
L^1(\bbR)\ast L^2(\bbR)\subset L^2(\bbR),  \quad 
L^2(\bbR)\ast L^2(\bbR)\subset L^{\infty}(\bbR).\no
\end{align}

For each $s\in \bbR$ and $p\in [1,\infty)$, the space of Bessel potentials $H^{s,p}(\bbR)$ is introduced as follows:
\begin{equation}
H^{s,p}(\bbR) = \bigg\{T\in \cS'(\bbR)\,\bigg|\, \Big[\big(1+|\cdot|^2\big)^{s/2}\widehat{T}\,\Big]^{\vee}
\in L^p(\bbR)\bigg\},\quad s\in \bbR,
\end{equation}
where the symbols\; $\widehat{}$\; and\; $\widecheck{}$\; denote the Fourier and inverse Fourier transforms of the distributions involved and $\big(1+|\cdot|^2\big)^{s/2}$ denotes the operator of multiplication by the smooth function $\bbR \ni x\mapsto \big(1+|x|^2\big)^{s/2}$.  The space $H^{s,p}(\bbR)$ is equipped with the norm $\|\cdot\|_{H^{s,p}(\bbR)}$ defined by
\begin{equation}
\|T\|_{H^{s,p}(\bbR)} = \bigg\|\Big[\big(1+|\cdot|^2\big)^{s/2}\widehat{T}\,\Big]^{\vee}\bigg\|_{L^p(\bbR)},\quad T\in H^{s,p}(\bbR),\, s\in \bbR.
\end{equation}

In the special case $p=2$, the shorthand notation $H^s(\bbR):=H^{s,2}(\bbR)$ will be used. Letting $H_0$ denote the self-adjoint realization of $-d^2/dx^2$ in $L^2(\bbR)$; that is,
\begin{equation}\lb{3.12c}
H_0f = - f'',\quad f\in \dom(H_0) = H^2(\bbR),
\end{equation}
so that $\dom\big(H_0^{1/2}\big) = H^1(\bbR)$, one verifies, upon comparison with the notation used in Section \ref{s2} (cf.\ \eqref{5.1}, \eqref{2.11}), that 
\begin{equation}
H^s(\bbR) = \cH_s\big(H_0^{1/2}\big), \quad s \in \bbR.
\end{equation}

 If $s>0$ and $f\in H^s(\bbR)$, then $\big(1+|\cdot|^2\big)^{s/2}\widehat{f}\in L^2(\bbR)$, so there exists $g\in L^2(\bbR)$ such that
\begin{equation}
{}_{\cS(\bbR)}\Big\langle \varphi, \big(1+|\cdot|^2\big)^{s/2}\widehat{f}\,\Big\rangle{}_{\cS'(\bbR)}= \int_{\bbR}\overline{\varphi(x)}g(x)\, dx,\quad \varphi\in \cS(\bbR).
\end{equation}
Furthermore, $\big(1+|\cdot|^2\big)^{-s/2}g\in L^2(\bbR)$ and
\begin{align}
\int_{\bbR} \overline{\varphi(x)}g(x) \big(1+|x|^2\big)^{-s/2}\, dx&= {}_{\cS(\bbR)}
\Big\langle \big(1+|\cdot|^2\big)^{-s/2}\varphi, \big(1+|\cdot|^2\big)^{s/2}\widehat{f}\,\Big\rangle{}_{\cS'(\bbR)}\no\\
&= {}_{\cS(\bbR)}\big\langle \varphi, \widehat{f}\,\big\rangle{}_{\cS'(\bbR)},\quad \varphi\in \cS(\bbR),
\end{align}
so that $\widehat{f} = \big(1+|\cdot|^2\big)^{-s/2}g\in L^2(\bbR)$ and, consequently, $f\in L^2(\bbR)$.  In addition,
\begin{align}
\|f\|_{L^2(\bbR)} = \big\|\widehat{f}\,\big\|_{L^2(\bbR)} \leq \Big\|\big(1+|\cdot|^2\big)^{s/2}\widehat{f}\,\Big\|_{L^2(\bbR)} = \|f\|_{H^s(\bbR)}.
\end{align}
Hence, one obtains the dense and continuous embedding 
\begin{equation}
H^s(\bbR) \hookrightarrow L^2(\bbR),\quad s\in (0,\infty),
\end{equation}
with the inequality
\begin{equation}
\|f\|_{L^2(\bbR)} \leq \|f\|_{H^s(\bbR)},\quad f\in H^s(\bbR).
\end{equation}

\begin{proposition}\lb{p2.1}
Let $n\in \bbN$.  If $\kappa\in (0,\infty)$, then
\begin{equation}
\big(H_0 + \kappa^2I_{L^2(\bbR)}\big)^{-1/2} \in \cB\big(H^{-1}(\bbR),L^2(\bbR)\big).
\end{equation}
\end{proposition}
\begin{proof}
In order to first show that $\big(H_0 + \kappa^2I_{L^2(\bbR)}\big)^{-1/2}$ maps from $H^{-1}(\bbR)$ to $L^2(\bbR)$, let $f\in H^{-1}(\bbR)$ so that $f\in \cS'(\bbR)$ and $\big(1+|\cdot|^2\big)^{-1/2}\widehat{f}\in L^2(\bbR)$.  Then
\begin{align}
\begin{split} 
&\Big[\big(H_0+\kappa^2I_{L^2(\bbR)}\big)^{-1/2}f\Big]\;\widehat{} 
 = \big(|\cdot|^2+\kappa^2\big)^{-1/2}\widehat{f}    \\
&\quad = \big(|\cdot|^2+\kappa^2\big)^{-1/2} \big(|\cdot|^2+1\big)^{1/2} \big(|\cdot|^2+1\big)^{-1/2}\widehat{f} \in L^2(\bbR),
\end{split} 
\end{align}
where the $L^2(\bbR)$ containment follows from the fact that $\big(|\dott|^2+1\big)^{-1/2}\widehat{f} \in L^2(\bbR)$ and the (smooth) function $\bbR\ni x\mapsto \big(|x|^2+\kappa^2\big)^{-1/2} 
\big(|x|^2+1\big)^{1/2}$ is bounded.  Therefore, $\big(H_0+\kappa^2I_{L^2(\bbR)}\big)^{-1/2}f\in L^2(\bbR)$ follows by taking inverse Fourier transforms.  The boundedness property is a consequence of Plancherel's theorem and the definition of $\|\cdot\|_{H^{-1}(\bbR)}$:
\begin{align}
\begin{split} 
& \Big\|\big(H_0+\kappa^2I_{L^2(\bbR)}\big)^{-1/2}f\Big\|_{L^2(\bbR)} = \Big\|\big(|\cdot|^2+\kappa^2\big)^{-1/2}\widehat{f}\,\Big\|_{L^2(\bbR)}     \\
& \quad \leq C(\kappa)\Big\|\big(|\cdot|^2+1\big)^{-1/2}\widehat{f}\,\Big\|_{L^2(\bbR)} 
= C(\kappa)\|f\|_{H^{-1}(\bbR)},\quad f\in H^{-1}(\bbR),
\end{split} 
\end{align}
where
\begin{equation}\lb{2.10}
C(\kappa):=\Big\|\big(|\cdot|^2+\kappa^2\big)^{-1/2} \big(|\cdot|^2+1\big)^{1/2}\Big\|_{L^{\infty}(\bbR)}=
\begin{cases}
\kappa^{-1},&\kappa\in (0,1),\\
1,&\kappa\in [1,\infty).
\end{cases}
\end{equation}
\end{proof}

\begin{proposition}\lb{p2.2}
Let $n\in \bbN$.  If $\kappa\in (0,\infty)$, then
\begin{equation}
\big(H_0 + \kappa^2I_{L^2(\bbR)}\big)^{-1/2} \in \cB\big(L^2(\bbR),H^1(\bbR)\big).
\end{equation}
\end{proposition}
\begin{proof}
To first show that $\big(H_0 + \kappa^2I_{L^2(\bbR)}\big)^{-1/2}$ maps from $L^2(\bbR)$ to $H^1(\bbR)$, let $f\in L^2(\bbR)$.  One then calculates
\begin{align}
& \big(|\cdot|^2+1\big)^{1/2}\Big[\big(H_0+\kappa^2I_{L^2(\bbR)}\big)^{-1/2}f\Big]\;\widehat{}\no\\
&\quad = \big(|\cdot|^2+1\big)^{1/2} \big(|\cdot|^2+\kappa^2\big)^{-1/2}\widehat{f}\in L^2(\bbR),
\end{align}
where the $L^2(\bbR)$ containment follows from the fact that $\widehat{f}\in L^2(\bbR)$ and the function $\bbR\ni x\mapsto \big(|x|^2+1\big)^{1/2} \big(|x|^2+\kappa^2\big)^{-1/2}$ is bounded.  Therefore, the inclusion $\big(H_0 + \kappa^2I_{L^2(\bbR)}\big)^{-1/2}f\in H^1(\bbR)$ holds.   The boundedness property follows from (repeated use of) Plancherel's theorem:
\begin{align}
& \Big\|\big(H_0 + \kappa^2I_{L^2(\bbR)}\big)^{-1/2}f \Big\|_{H^1(\bbR)} = \Big\|\big(|\cdot|^2+1\big)^{1/2} \big(|\cdot|^2+\kappa^2\big)^{-1/2}\widehat{f}\, \Big\|_{L^2(\bbR)}\no\\
&\quad \leq C(\kappa)\big\|\widehat{f}\,\big\|_{L^2(\bbR)} = C(\kappa)\|f\|_{L^2(\bbR)},\quad f\in L^2(\bbR),
\end{align}
with $C(\kappa)$ defined by \eqref{2.10}.
\end{proof}

\begin{remark}
With only minor modification, the same arguments used to prove Propositions \ref{p2.1} and \ref{p2.2} show that
\begin{equation}\lb{2.18aa}
\begin{split}
&\big(H_0 - zI_{L^2(\bbR)}\big)^{-1/2} \in \cB\big(H^{-1}(\bbR),L^2(\bbR)\big),\\
&\big(H_0 - zI_{L^2(\bbR)}\big)^{-1/2} \in \cB\big(L^2(\bbR),H^1(\bbR)\big),\quad z\in \bbC\backslash[0,\infty).
\end{split}
\end{equation}
In particular, 
\begin{equation}
H^1(\bbR) \hookrightarrow L^2(\bbR) = L^2(\bbR)^* \hookrightarrow H^{-1}(\bbR) = \big[H^1(\bbR)\big]^*. 
\end{equation}
\end{remark}

Next, we recall the concept of a Sobolev multiplier from $H^1(\bbR)$ to $H^{-1}(\bbR)$.  If $V\in \cS'(\bbR)$, then $V$ is a multiplier for $\cS(\bbR)$ in the following sense:  If $\varphi\in \cS(\bbR)$, then $V\varphi\in \cS'(\bbR)$, where $V\varphi$ is defined according to
\begin{align}\lb{2.18a}
{}_{\cS(\bbR)}\langle \psi, V\varphi \rangle_{\cS'(\bbR)} := {}_{\cS(\bbR)}\langle \overline{\varphi}\psi, V\rangle_{\cS'(\bbR)},\quad \psi \in \cS(\bbR).
\end{align}
In particular, $\mathfrak{q}_V$ defined by
\begin{equation}\lb{2.19aaaa}
\mathfrak{q}_V[\psi,\varphi] := {}_{\cS(\bbR)}\langle \psi, V\varphi \rangle_{\cS'(\bbR)},\quad \varphi,\psi \in \cS(\bbR),
\end{equation}
is a well-defined sesquilinear form on $\cS(\bbR)$.

\begin{definition}\lb{def_multiplier}
The tempered distribution $V\in \cS'(\bbR)$ is called a Sobolev multiplier from $H^1(\bbR)$ to $H^{-1}(\bbR)$ if $\mathfrak{q}_V$, defined on $\cS(\bbR)$ by \eqref{2.19aaaa}, continuously extends from $\cS(\bbR)$ to $H^1(\bbR)$; that is, if
\begin{equation}\lb{2.18}
|\mathfrak{q}_V[\psi,\varphi]| \leq C_V \|\psi\|_{H^1(\bbR)}\|\varphi\|_{H^1(\bbR)},\quad \varphi,\psi\in \cS(\bbR),
\end{equation}
for some $C_V\in (0,\infty)$.
\end{definition}

One notes that when \eqref{2.18} holds, $\mathfrak{q}_V$ may be extended to a bounded sesquilinear form, which we denote by $\widetilde{\mathfrak{q}}_V$, on $H^1(\bbR)$ by taking
\begin{equation}\lb{2.19}
\widetilde{\mathfrak{q}}_V[g,f] := \lim_{n\to \infty} \mathfrak{q}_V[\psi_n,\varphi_n],\quad f,g\in H^1(\bbR),
\end{equation}
where $\{\varphi_n\}_{n=1}^{\infty},\{\psi_n\}_{n=1}^{\infty}\subset \cS(\bbR)$ are any sequences such that
\begin{equation}
\lim_{n\to \infty}\|\varphi_n-f\|_{H^1(\bbR)} = \lim_{n\to \infty}\|\psi_n-g\|_{H^1(\bbR)} = 0.
\end{equation}
The definition in \eqref{2.19} is independent of the sequences $\{\varphi_n\}_{n=1}^{\infty},\{\psi_n\}_{n=1}^{\infty}\subset \cS(\bbR)$ chosen and \eqref{2.18} remains valid for the extension:
\begin{equation}\lb{2.21}
|\widetilde{\mathfrak{q}}_V[g,f]| \leq C_V \|g\|_{H^1(\bbR)}\|f\|_{H^1(\bbR)},\quad f,g\in H^1(\bbR).
\end{equation}
In the following, the symbol $\cM\big(H^1(\bbR),H^{-1}(\bbR)\big)$ denotes the set of all Sobolev multipliers from $H^1(\bbR)$ to $H^{-1}(\bbR)$.

If $V\in \cM\big(H^1(\bbR),H^{-1}(\bbR)\big)$, then the distributional pairing on the left-hand side in \eqref{2.18a} extends to $H^1(\bbR)$ as follows:
\begin{equation}
{}_{H^1(\bbR)}\langle g,Vf\rangle{}_{H^{-1}(\bbR)} = \widetilde{\mathfrak{q}}_V[g,f],\quad f,g\in H^1(\bbR).
\end{equation}
The multiplier operator defined by
\begin{equation}
H^1(\bbR)\ni f \mapsto Vf\in H^{-1}(\bbR)
\end{equation}
is linear and bounded by \eqref{2.21} since
\begin{equation}\lb{2.25a}
\|Vf\|_{H^{-1}(\bbR)} = \sup_{g\in H^1(\bbR)\backslash\{0\}}\frac{|\wti{\mathfrak{q}}_V[g,f]|}{\|g\|_{H^1(\bbR)}} \leq C_V\|f\|_{H^1(\bbR)},\quad f\in H^1(\bbR).
\end{equation}
In particular, the operator of multiplication by $V$ is bounded from $H^1(\bbR)$ to $H^{-1}(\bbR)$ with
\begin{equation}\lb{2.26a}
\|V\|_{\cB(H^1(\bbR),H^{-1}(\bbR))}\leq C_V.
\end{equation}
We note that, in a slight abuse of notation, we use the symbol ``$V$'' in \eqref{2.26a} (and elsewhere) to denote the operator of multiplication by $V\in \cM\big(H^1(\bbR),H^{-1}(\bbR)\big)$.

Every $V\in H^{-1}(\bbR)$ is a Sobolev multiplier from $H^1(\bbR)$ to $H^{-1}(\bbR)$; that is, $V$ belongs to $\cM\big(H^1(\bbR),H^{-1}(\bbR)\big)$.  This fact, which we prove for completeness below in Proposition \ref{p2.7}, relies on the following algebraic property of the $H^1(\bbR)$ norm.

\begin{lemma}\lb{l2.7a}
If $\varphi,\psi\in \cS(\bbR)$, then $\|\varphi\psi\|_{H^1(\bbR)}^2\leq 2\|\varphi\|_{H^1(\bbR)}^2\|\psi\|_{H^1(\bbR)}^2$.
\end{lemma}
\begin{proof}
The one-dimensional Sobolev inequality (see, e.g., \cite[Theorem 8.5]{LL01}) yields:
\begin{equation}\lb{2.36aa}
2\|\varphi\|_{L^{\infty}(\bbR)}^2\leq \|\varphi\|_{H^1(\bbR)}^2,\quad \varphi\in \cS(\bbR).
\end{equation}
If $\varphi,\psi\in \cS(\bbR)$, then basic estimates combined with \eqref{2.36aa} yield:
\begin{align}
\|\varphi\psi\|_{H^1(\bbR)}^2 &= \int_{\bbR} \big\{ |\varphi\psi|^2 + |(\varphi\psi)'|^2\big\}\, dx\no\\ 
&\leq \int_{\bbR}\big\{2|\psi|^2\big[|\varphi|^2+|\varphi'|^2 \big] + 2|\varphi|^2\big[|\psi|^2+|\psi'|^2 \big] \big\}\, dx\no\\
&\leq \int_{\bbR}\big\{\|\psi\|_{H^1(\bbR)}^2\big[|\varphi|^2+|\varphi'|^2 \big] + \|\varphi\|_{H^1(\bbR)}^2\big[|\psi|^2+|\psi'|^2 \big] \big\}\, dx\no\\
&= 2\|\varphi\|_{H^1(\bbR)}^2\|\psi\|_{H^1(\bbR)}^2.
\end{align}
\end{proof}

\begin{proposition}\lb{p2.7}
If $V\in H^{-1}(\bbR)$, then $V\in \cM\big(H^1(\bbR),H^{-1}(\bbR)\big)$ and
\begin{equation}\lb{2.35a}
\|V\|_{\cB(H^1(\bbR),H^{-1}(\bbR))}\leq 2^{1/2}\|V\|_{H^{-1}(\bbR)}.
\end{equation}
\end{proposition}
\begin{proof}
If $\varphi,\psi\in \cS(\bbR)$, then Lemma \ref{l2.7a} (applied to $\overline{\psi}$ and $\varphi$) implies
\begin{align}
|\mathfrak{q}_V[\psi,\varphi]| &= \big|{}_{H^1(\bbR)}\langle \overline{\varphi}\psi,V\rangle_{H^{-1}(\bbR)}\big| \leq \|V\|_{H^{-1}(\bbR)}\|\overline{\varphi}\psi\|_{H^1(\bbR)}\no\\
&\leq 2^{1/2} \|V\|_{H^{-1}(\bbR)} \|\varphi\|_{H^1(\bbR)} \|\psi\|_{H^1(\bbR)}.
\end{align}
Hence, \eqref{2.18} holds with $C_V=2^{1/2}\|V\|_{H^{-1}(\bbR)}$.  Finally, \eqref{2.35a} follows from \eqref{2.26a}.
\end{proof}

For additional results on multipliers on Sobolev spaces see, for instance, \cite{BS17}, \cite{BS19}, \cite{NS99}, \cite{NS06}. 

Proposition \ref{p2.7} provides an elementary sufficient condition for a distribution to be a Sobolev multiplier from $H^1(\bbR)$ to $H^{-1}(\bbR)$. Recalling the notion of locally uniformly $L^p$-integrable, $p\in [1,\infty)$,
functions on $\bbR$,
\begin{equation}\lb{4.5}
\begin{split}
L^p_{\locunif}(\bbR; dx) = \bigg\{f \in L^p_{\loc}(\bbR; dx) \,\bigg|\,
\sup_{a\in\bbR} \bigg(\int_a^{a+1} dx \, |f(x)|^p\bigg) < \infty\bigg\},&\\
p\in[1,\infty),&
\end{split}
\end{equation} 
a complete characterization of the Sobolev multipliers from $H^1(\bbR)$ to $H^{-1}(\bbR)$ is given as follows (see also the summary in \cite[Sect.~3]{GW14}).

\begin{theorem}[{\cite{BS02}, \cite[Sects.~2.5, 11.4]{MS09}, \cite{MV02a}, \cite{MV05}}]
The following statements $(i)$--$(iii)$ hold:\\[1mm]
$(i)$ $V \in \cS'(\bbR)$ belongs to $\cM\big(H^1(\bbR),H^{-1}(\bbR)\big)$ if and only if $V$ is of the form
\begin{equation}
V=q_1+q_2',\quad \text{where $q_j\in L_{\locunif}^j(\bbR)$, $j=1,2$}.    \lb{4.6}
\end{equation}
In addition to \eqref{4.6}, $V$ is compact as a Sobolev multiplier from $H^1(\bbR)$ to $H^{-1}(\bbR)$ if and only if
\begin{equation}
\lim_{|a|\to \infty} \int_a^{a+1} |q_j(x)|^j\, dx = 0,\quad j=1,2.
\end{equation}
$(ii)$ $V\in \cS'(\bbR)$ belongs to $\cM\big(H^1(\bbR),H^{-1}(\bbR)\big)$ if and only if $V$ is of the form
\begin{equation}
V=q_\infty+q_2',\quad \text{where $q_\infty\in L^{\infty}(\bbR),\, q_2\in L_{\locunif}^2(\bbR)$}.
\end{equation}
$(iii)$ $V\in \cS'(\bbR)$ belongs to $\cM\big(H^1(\bbR),H^{-1}(\bbR)\big)$ if and only if $V$ is of the form
\begin{equation}
V=q_0+q_2',\quad \text{where $q_0, q_2\in L_{\loc}^2(\bbR)$}.
\end{equation}
\end{theorem}

By the results of Propositions \ref{p2.1} and \ref{p2.2}, if $V\in \cS'(\bbR)$ is a Sobolev multiplier from $H^1(\bbR)$ to $H^{-1}(\bbR)$, then
\begin{equation}\lb{2.32}
\big(H_0+\kappa^2I_{L^2(\bbR)}\big)^{-1/2}V\big(H_0+\kappa^2I_{L^2(\bbR)}\big)^{-1/2}\in \cB\big(L^2(\bbR)\big),\quad \kappa\in (0,\infty).
\end{equation}

It is clear that all results in this appendix hold for $\bbR$ replaced by $\bbR^n$, $n\in \bbN$, upon suitable modifications. 
 
\medskip
\noindent 
{\bf Acknowledgments.} We are indebted to the referee for very helpful comments and hints to the literature. 

\medskip


 
\end{document}